\documentclass{amsart}

\usepackage[all]{xy}

\usepackage{amsmath,amssymb,amsthm,amsfonts,enumerate,array,color,lscape,fancyhdr,layout,pst-all}
\usepackage[english]{babel}
\usepackage[latin1]{inputenc}
\usepackage{young}

\newcounter{numberofremark}
\newcommand\nothing[1]{}
\usepackage[hyperindex,bookmarksnumbered,plainpages]{hyperref}

\newcommand{\dcl}{\DeclareMathOperator}
\dcl\C{\mathbb C}
\dcl\g{\mathfrak{g}}
\dcl\gl{\mathfrak{gl}}
\dcl\gr{gr}
\dcl\gts{\mathfrak{G}}
\dcl\K{\mathbb K}
\dcl\NN{\mathbb{N}}
\dcl\qdet{qdet}
\dcl\Sp{Specm}
\dcl\tr{\textup{tr}\ }
\dcl\wk{\mathbb{W}}
\dcl\Z{\mathbb{Z}}
\dcl\Y{Y}
\dcl\U{U}
\dcl\T{T}

\renewcommand\k{{\Bbbk}}

\newlength\yStones
\newlength\xStones
\newlength\xxStones

\makeatletter
\def\Stones{\pst@object{Stones}}
\def\Stones@i#1{%
  \pst@killglue%
  \begingroup%
  \use@par%
  \setlength\xxStones{\xStones}%
  \expandafter\Stones@ii#1,,\@nil
  \endgroup
  \global\addtolength\xStones{0.6cm}%
  \global\addtolength\yStones{-7.5mm}}%
\def\Stones@ii#1,#2,#3\@nil{%
  \rput(\xxStones,\yStones){%
    \psframebox[framesep=0]{%
      \parbox[c][6mm][c]{11mm}{\makebox[11mm]{$#1$}}}}%
  \addtolength\xxStones{1.2cm}%
  \ifx\relax#2\relax\else\Stones@ii#2,#3\@nil\fi}
\makeatother

\def\Stone#1{\fbox{\makebox[8mm]{\strut#1}}\kern2pt}

\newtheorem{theorem}{Theorem}[section]
\newtheorem{lemma}[theorem]{Lemma}
\newtheorem{corollary}[theorem]{Corollary}
\newtheorem{proposition}[theorem]{Proposition}
\newtheorem{example}[theorem]{Example}
\newtheorem{remark}[theorem]{Remark}
\newtheorem{notation}[theorem]{Notation}

\newtheorem{definition}[theorem]{Definition}

\begin{document}
\allowdisplaybreaks

\title{Gelfand-Tsetlin Varieties for Yangians}

\author{Germán Benitez Monsalve}
\address{\noindent Departamento de Matem\'atica, Instituto de Ciências Exatas, Universidade Federal do
Amazonas,  Manaus AM, Brazil}
\email{gabm03@gmail.com}


\begin{abstract}

S. Ovsienko proved that the Gelfand-Tsetlin variety for $\mathfrak{gl}_n$ is equidimensional (i.e., all its irreducible components have the same dimension) of dimension $\frac{n(n-1)}{2}$. This result is known as \textit{Ovsienko's Theorem} and it has important consequences in Representation Theory of Algebras. In this paper, we will study the Gelfand-Tsetlin varieties for restricted Yangian $Y_p(\mathfrak{gl}_n)$ and the equidimensionality for $Y_p(\mathfrak{gl}_3)$.

\end{abstract}

\maketitle



\tableofcontents    


\section{Introduction}

We consider the Gelfand-Tsetlin variety associated to the Gelfand-Tsetlin subalgebra $\Gamma$ for General Linear Lie algebra $\gl_n$ over an algebraically closed field $\k$. In \cite{Ovs} S. Ovsienko has proved that such variety is equidimensional (i.e., its all irreducible components has the same dimension) of dimension $\frac{n(n-1)}{2}$. That result is known as \textit{``Ovsienko's Theorem''} and, as a consequences, the Gelfand-Tsetlin variety is a complete intersection. 

\vspace{0.1cm}

A well-known Kostant's Theorem says that the Universal Enveloping algebra $U(\mathfrak{g})$ of a semisimple Lie algebra $\mathfrak{g}$ over $\C$,  is a free module over its center (Theorem ($0.13$) in \cite{K}). When $\mathfrak{g}=\gl_n$, the variety associated to the center of the universal
enveloping algebra $U(\gl_n)$ coincides with the variety of nilpotent matrices, which is irreducible of dimension $n^2-n$. Also, S. Ovsienko proved in \cite{O} that $U(\gl_n)$ is free as a left (right) module over its Gelfand-Tsetlin subalgebra.

 A generalization of Kostant's Theorem in \cite{K} was proved by V. Futorny and S. Ovsienko in \cite{FutOvs05} for a class of special filtered algebras, where they have stablished the following kye result:

\vspace{0.1cm}

\noindent \textbf{Theorem:}
\textit{Let $U$ be a special filtered algebra, $g_1,g_2,\ldots,g_t\in U$ mutually commu\-ting elements such that its graded images determine a complete intersection for the graded algebra associated to $U$, and let $\Gamma=\k[g_1,g_2,\ldots,g_t]$. Then $U$ is a free left (right) $\Gamma$-module.}  

\vspace{0.1cm}

As a consequence of this theorem the authors in \cite{FutOvs05} showed that restricted Yangian $Y_p(\gl_n)$ of level $p$ for $\gl_n$ is free over its center. For the full Yangian $Y(\gl_n)$ this was shown previously by  A. Molev, M. Nazarov and G. Olshanski\u{i} in \cite{MNO}. In \cite{FMO} V. Futorny, A. Molev\ and\ S. Ovsienko showed that the finite $W$ algebras associated with $\gl_2$ (in particular, the restricted Yangian  $Y_p(\gl_2)$ of level $p$) is free over its Gelfand-Tsetlin subalgebra.

Gelfand-Tsetlin algebras appears in several contexts. Such algebras are related to important problems in representation theory of Lie algebras, for instances, E. B. Fomenko and A. S. Mischenko in \cite{FomMis} related such algebras in connection with the solutions of the Euler equation. Also, {\`E}. B. Vinberg in \cite{Vin} related the Gelfand-Tsetlin subalgebras in connection with subalgebras of ma\-xi\-mal Gelfand-Kirillov dimension of the universal enveloping algebra of a simple Lie algebra. B. Kostant and N. Wallach in {\cite{KW-1} and \cite{KW-2} used these subalgebras in connection with classical mechanics.

\vspace{0.5cm}

In this paper, we will study the Gelfand-Tsetlin variety for restricted Yangian $Y_p(\gl_n)$ and its equidimensio\-na\-lity. The paper is organized as follows: In section $\mathsection 2$ we introduce all necessary definitions, notations and results used throughout the paper, such as, equidimensionality of varieties and regular sequences. In section $\mathsection 3$ we recall the definition of Gelfand-Tsetlin subalgebras, Gelfand-Tsetlin varieties, Ovsienko's Theorem, its Weak Version (Theorem \eqref{fracagln}) and some consequences.

In section $\mathsection 4$ we define the restricted Yangian $Y_p(\gl_n)$ of level $p$ for $\gl_n$, its Gelfand-Tsetlin subalgebra and its Gelfand-Tsetlin variety
$\gts$. We prove the equidimensionality of the Gelfand-Tsetlin variety $\gts$ for $Y_2(\gl_3)$ (Proposition \eqref{Y2(3)}) exhi\-biting its decomposition in irreducible components. Also, we show in the Proposition \eqref{decomp1} a decomposition (it is not in irreducible components) of the Gelfand-Tsetlin variety for $Y_p(\gl_n)$ in the form
  $$\gts=\gts_1\cup\gts_2\cup\cdots\cup\gts_{p+1}.$$
In the Corollary \eqref{decompweak}, we show a decomposition of $\gts_1$ for $Y_p(\gl_n)$ inn the form
 $$\gts_1=\wk_{1}\cup\wk_{2},\ \ \mbox{if}\ \ p=1,2$$
 $$\gts_1=\wk^{p}\cup\wk_{1}\cup\wk_{2},\ \ \mbox{if}\ \ p\geq 3.$$
When $n=3$, the variety $\gts_1$ will be called \textit{Weak Version of the Gelfand-Tsetlin variety of} $\gts$ for $Y_p(\gl_3)$. Finally, we prove that the weak version $\gts_1$ of the Gelfand-Tsetlin variety $\gts$ for $Y_p(\gl_3)$ is equidimensional of dimension $\dim\gts_1=3p$. This is the aim result of this paper (Theorem  \eqref{G1}).

\vspace{0.5cm}

\noindent{\bf Acknowledgements.} The author wishes to express his gratitude to Vyacheslav Futorny for suggesting the problem and for many useful discussions during this
project.


\section{Preliminaries}

Throughout the paper we fix an algebraically closed field $\k$ of characteristic zero.  

For a \textit{reduced} (without nilpotent elements) \textit{affine} $\k-$\textit{algebra} $\Lambda$ (that is, an associative and commutative finitely gene\-rated $\k-$algebra), we denote by $\Sp\Lambda$ the variety of all maximal ideals of $\Lambda$. If $\Lambda$ is a polynomial algebra in $n$ variables, then we identify $\Sp\Lambda$ with $\k^n$. For an ideal $I\subseteq\Lambda$, denote by $V(I)\subseteq\Sp\Lambda$, the set of all zeroes of $I$, called
\textit{variety}. If $I$ is generated by $g_1,g_2,\ldots,g_r$, then we write $I=(g_1,g_2,\ldots,g_r)$ and $V(I)=V(g_1,g_2,\ldots,g_r)$. A variety $V$ is an \textit{equidimensional variety} if all its irreducible components have the same dimension and we denote by $\dim V$ the dimension of $V$.

We will denote the center of an associative algebra $\mathcal A$ by $Z(\mathcal A)$, the graded algebra of $\mathcal A$ by $\gr\mathcal(A)$ and the Universal enveloping algebra of a Lie algebra $\g$ by $U(\g)$.


A sequence $g_1,g_2,\ldots,g_t$ in a ring $R$ is called \emph{regular} if, the image class of $g_i$ is not a zero divisor, and is not invertible in $R/(g_1,g_2,\ldots,g_{i-1})$ for any $i=1,2,\ldots,t$.  


\begin{proposition}
\label{regperm}
Let $R$ be a noetherian ring and $g_1,g_2,\ldots,g_t$ a regular sequence of 
$R$. If $R$ is a graded ring and each $g_i$ is homogeneous of positive degree, then any permutation of $g_1,g_2,\ldots,g_t$ is regular in $R$.
\end{proposition}


\begin{proof} 
Theorem ($28$) in \cite{Mats70} (page $102$) or \cite{Mats89} (page $127$).

\end{proof}


\begin{corollary}
\label{regsubseq}
Let $R$ be a noetherian ring and $g_1,g_2,\ldots,g_t\in R$ a regular sequence of $R$. If $R$ is a graded ring and each $g_i$ is homogeneous of positive degree, then any subsequence of $g_1,g_2,\ldots,g_t$ is regular in $R$.
\end{corollary}


\begin{proof} 
By Proposition \eqref{regperm} and the definition of a regular sequence.

\end{proof}

\begin{proposition}
\label{regvseq}
Let $R$ be an affine algebra of Krull dimension $n$, and 
$g_1,g_2,\ldots,g_t$ be a sequence of elements in $R$ with $0\leq t\leq n$. 

\begin{enumerate}[i.]
 \item 
If $R$ is graded and $g_1,g_2,\ldots,g_t$ are homogeneous, then
$g_1,g_2,\ldots,g_t$ is regular in $R$ if and only if the sequence
$g_1,g_2,\ldots,g_t$ is a complete intersection for $R$. 
 
 \item If $R$ is a Cohen-Macaulay algebra, then the sequence $g_1,g_2,\ldots,g_t$ is a complete intersection for $R$ if and only if the variety $V(g_1,g_2,\ldots,g_t)$ is equidimensional of dimension $n-t$.
\end{enumerate}
\end{proposition}


\begin{proof} 
Proposition ($2.1$) in \cite{FutOvs05}.

\end{proof}


\begin{proposition}
\label{projhyper}
Let $R=\k[X_1,X_2,\ldots,X_n]$ be a polynomial algebra, and let
$G_1,G_2,\ldots,G_t\in R$. The sequence 
$$X_1,X_2,\ldots,X_r,G_1,G_2,\ldots,G_t$$
is a complete intersection for $R$ if and only if the sequence
$g_1,g_2,\ldots,g_t$ is a complete intersection for
$\k[X_{r+1},X_{r+2},\ldots,X_{n}]$, where 
$$g_i(X_{r+1},X_{r+2},\ldots,X_{n})=G_i(0,0,\ldots,0,X_{r+1},X_{r+2},\ldots,X_{n}),\ \ \forall i=1,2,\ldots,t.$$
\end{proposition}


\begin{proof} 
Lemma ($2.2$) in \cite{FutOvs05}.

\end{proof}



\section{Ovsienko's Theorem and a weak version}


From now on fix some integer $n\geq 1$. 


\subsection{Gelfand-Tsetlin variety}
For $m\in\left\{1,2,\ldots,n\right\}$, let $\gl_m$ denote the gene\-ral linear Lie algebra of all $m\times m$ matrices over $\k$ with commutator pro\-duct and the
standard basis $\{E_{ij}\mid 1\leq i,j \leq m\}$ of matrix units and denote for $Z_m:=Z(U(\gl_m))$ the center of $U(\gl_m)$.

Clearly, $\gl_m$ for each $m\in\left\{1,2,\ldots,n-1\right\}$ is a Lie subalgebra of $\gl_{m+1}$ and $\gl_n$. Therefore, there exists a chain of Lie subalgebras of $\gl_n$
$$\gl_1\subset\gl_2\subset\cdots\subset\gl_{n-1}\subset\gl_n$$
and the induced chain of subalgebras of $U(\gl_n)$
$$U(\gl_1)\subset U(\gl_2)\subset\cdots\subset U(\gl_{n-1})\subset U(\gl_n).$$


\begin{definition}
The \textbf{Gelfand-Tsetlin subalgebra $\Gamma$ of} $U(\gl_n)$ is defined to be  the subalgebra $\Gamma$ of $U(\gl_n)$ generated by $\left\{Z_1,Z_2,\ldots,Z_n\right\}$.
\end{definition}


\begin{proposition}[\v Zelobenko, 1973]
For any $m\in\left\{1,2,\ldots,n\right\}$, the center $Z_m$ is a polynomial algebra in $m$ variables $\left\{\gamma_{mj}\ :\ j=1,2,\ldots,m\right\}$, with
 $$\gamma_{ij}=\sum_{t_1,t_2,\ldots,t_j\in \left\{1,2,\ldots,i\right\}}{E_{t_1t_2}E_{t_2t_3}\cdots E_{t_{j-1}t_j}E_{t_jt_1}}.$$
The subalgebra $\Gamma$ is a polynomial algebra in $\frac{n(n+1)}{2}$ variables $\left\{\gamma_{ij}\ :\ 1\leq j\leq i\leq n\right\}$.
\end{proposition}


\begin{proof}
 See \cite{Zelobenko} page $169$.

\end{proof}


\begin{remark}
\label{contatraco}
Consider the matrix 

$$E:=\left(\begin{matrix}
 E_{11} & E_{12} & \cdots & E_{1n}\\
 E_{21} & E_{22} & \cdots & E_{2n}\\
 \vdots & \vdots & \ddots & \vdots\\
 E_{n1} & E_{n2} & \cdots & E_{nn}
\end{matrix}\right)\in M_n\left(U(\gl_n)\right)$$
and for each  $i\in\left\{1,2,\ldots,n\right\}$ the $i-$th principal submatrix

$$E_i:=\left(\begin{matrix}
 E_{11} & E_{12} & \cdots & E_{1i}\\
 E_{21} & E_{22} & \cdots & E_{2i}\\
 \vdots & \vdots & \ddots & \vdots\\
 E_{i1} & E_{i2} & \cdots & E_{ii}
\end{matrix}\right)\in M_i\left(U(\gl_i)\right)$$ 
Therefore, for $1\leq j\leq i\leq n$
$$\gamma_{ij}=\tr(E_i^j)=\sum_{t_1=1}^{i}\sum_{t_2=1}^{i}\cdots\sum_{t_j=1}^{i}{E_{t_1t_2}E_{t_2t_3}\cdots E_{t_{j-1}t_{j}}E_{t_jt_1}}$$

\end{remark}


\begin{remark}
 For others generators of the Gelfand-Tsetlin subalgebra $\Gamma$ see \cite{Gelfandetal} or \cite{Molevbook} (pages $246-250$).
\end{remark}



Clearly, the element $\gamma_{ij}$ can be viewed as a polynomial in noncommutative variables $E_{ij}$. But, by the Poincare-Birkhoff-Witt theo\-rem the graded algebra $\gr(U(\gl_n))$ is a polynomial algebra in variables 
$\overline{E}_{ij}$ with $i,j=1,2,\ldots ,n$ 
 $$\gr(U(\gl(n)))\cong\k\left[\overline{E}_{ij} : i,j=1,2,\ldots ,n\right]$$  

Therefore, the elements $\overline{\gamma}_{ij}$ are polynomials in commutative variables $\overline{E}_{ij}$ and considering the notation $X_{ij}:=\overline{E}_{ij}$ for $i,j=1,2,\ldots ,n$ then  
 $$\overline{\gamma}_{ij}=\sum_{t_1,t_2,\ldots,t_j\in \left\{1,2,\ldots,i\right\}}{X_{t_1t_2}X_{t_2t_3}\cdots X_{t_{j-1}t_j}X_{t_jt_1}}.$$


\begin{definition}
The \textbf{Gelfand-Tsetlin variety for} $\gl_n$ is the algebraic variety 
$$V\left(\left\{\overline{\gamma}_{ij} : i=1,2,\ldots ,n;\ j=1,2,\ldots ,i \right\}\right)\subset \k^{n^2}$$
\end{definition}
 
 
\begin{theorem}[\textbf{Ovsienko's Theorem}, $2003$]
\label{tmaOvs}
\noindent

The Gelfand-Tsetlin variety for $\gl_n$ is equidimensional of dimension $\frac{n(n-1)}{2}$.
\end{theorem}
 
 
\begin{proof}
 See \cite{Ovs}.
 
\end{proof}


\subsection{Weak version of the Ovsienko's theorem}


\begin{theorem}[\textbf{Weak version of the Ovsienko's theorem}]
\label{fracagln}
\noindent

The variety $V_n:=V\left(\sigma_{n2},\sigma_{n3},\sigma_{n4},\ldots,\sigma_{nn}\right)\subset\k^{\frac{(n+2)(n-1)}{2}}$ is equidimensional of dimensão $\dim\left(V_n\right)=\frac{n(n-1)}{2}$, where
 $$\sigma_{nj}=\sum_{n>t_1>t_2>\cdots>t_{j-1}\geq 1}{X_{nt_1}X_{t_1t_2}\cdots X_{t_{j-2}t_{j-1}}X_{t_{j-1}n}};\ \ \ j=2,3,\ldots,n.$$
\end{theorem}


\begin{proof}
Theorem (3.9) in \cite{gbm}. 

\end{proof}


\begin{remark}
Note that:
\begin{enumerate}[a.]
  \item $V_2$ is exactly the
Gelfand-Tsetlin variety for $\gl_2$ and the Gelfand-Tsetlin variety for $\gl_3$ is the union between $V_3$ and other subvariety which isomorphic to $V_3$. The name \textit{Weak version} is because $V_n$ is a subvariety of the Gelfand-Tsetlin variety.

\item All regular subvarieties of Gelfand-Tsetlin variety for $\gl_n$ are $V\left(f_1,f_2,\ldots,f_t\right)\subset\k^{n^2}$, where each polynomial is a variable, that is, for each $i$
$$f_i=X_{rs}\ \ ,\ \ \mbox{for some}\ \ 1\leq r,s\leq n.$$
Without loss of generality, we will assume that  $f_i\neq f_j\ ,\ \forall i\neq j.$
\end{enumerate}

\end{remark}


Independently, B. Kostant and N. Wallach in \cite{KW-1} and \cite{KW-2} guarantee that, all the regular components of the Gelfand-Tsetlin variety for $\gl_n$ are equidimensionals with dimension $\frac{n(n-1)}{2}$. But, that result is a corollary of the weak version:


\begin{corollary}
\label{corVFraca}
All the regular components of the Gelfand-Tsetlin variety are isomorphics. In particular, these regular components and 
$$V_{\leq}:=V\left(\left\{X_{ij}:1\leq i\leq j\leq n\right\}\right)$$
are isomorphics.
\end{corollary}


\begin{proof}
Corollary (3.12) in \cite{gbm}.

\end{proof}


\subsection{Kostant-Wallach map vs Gelfand-Tseltin varieties}

We can view the fiber in zero of the Kostant-Wallach map and its partial maps as Gelfand-Tsetlin varieties, the following form:  

\begin{definition}
\noindent
\begin{enumerate}
 \item The \textbf{Kostant-Wallach map} is the morphism given by 
$$\begin{array}{rccl}
 \Phi: & M_n(\mathbb{C}) & \longrightarrow & \mathbb{C}^{\frac{n(n+1)}{2}}\\
       & X               & \longmapsto     & \Phi(X):=\left(\chi_1(X),\chi_2(X),\ldots,\chi_n(X)\right).
 \end{array}.$$

\item For each $k=1,2,\dots,n$ the \textbf{$k$-partial Kostant-Wallach map} is the morphism 
 $$\Phi_k:=\pi_k\circ\Phi,$$ 
 where $\pi_k:\C\times\C^2\times\cdots\C^n\longrightarrow\C^{n-k+1}\times\cdots\times\C^{n-1}\times\C^{n}$ 
is the projection on the last $k$ factors.

 \end{enumerate}
 
\end{definition}


\begin{definition}
For $k=1,2,\ldots,n$ and $\beta\in\C^{n-k+1}\times\cdots\times\C^{n-1}\times\C^{n}$, we define the $k$-\textbf{partial Gelfand-Tsetlin variety in} $\beta$ as the
algebraic variety
$$\widetilde{V}_{\beta}^k=V\left(\left\{\overline{\gamma}_{ij}-\beta_{ij}:\ n-k+1\leq i \leq n\ \ \mbox{and}\ \ \ 1\leq j\leq i \right\}\right)\subset\C^{n^2}.$$

\end{definition}


\begin{remark}
Clearly $\Phi_n=\Phi$ and $\widetilde{V}_{0}^n$ is the Gelfand-Tsetlin variety for $\gl_n$.

\end{remark}


The relation between Kostant-Wallach map and Gelfand-Tseltin variety is determinated by the following proposition.


\begin{proposition}

\label{KWvsGTs}
For all $k=1,2,\dots,n$, the $k$-partial Gelfand-Tsetlin variety in zero  coincide with the fiber in zero of the $k$-partial Kostant-Wallach map, i.e., 
$$\widetilde{V}_{0}^k=\Phi_k^{-1}(0).$$
\end{proposition}


\begin{proof} Corollary (4.10) in \cite{gbm}.

\end{proof}


\begin{remark}
This proposition asserts that the Gelfand-Tsetlin variety coincide with the fiber in zero of the Kostant-Wallach map $\Phi^{-1}(0)$. As consequence of that fact, the Gelfand-Tsetlin variety for $\gl_n(\C)$ is exactly the set of the strongly nilpotent matrices $n\times n$ (where, a matrix $X\in M_n(\C)$ is say \textbf{strongly nilpotent}, when all its $i$-th principal submatrices $X_i$ are nilpotents).   
\end{remark}


\section{Restricted Yangians}

Let $p$ and $n$ be positive integers. The \textit{level} $p$ \textit{Yangian}  $Y_p(\gl_n)$ for the Lie algebra $\gl_n$ \cite{CHE, DRI} can be defined as the associative algebra over $\k$ with generators $t_{ij}^{(1)},t_{ij}^{(2)},\ldots,t_{ij}^{(p)}$, $i,j=1,2,\ldots,n$ subject to the relations 

\begin{equation*}
\label{genYan}
\left[T_{ij}(u),T_{kl}(v)\right]=\frac{1}{u-v}\left(T_{kj}(u)T_{il}(v)-T_{kj}(v)T_{il}(u)\right),
\end{equation*} 
where $u,v$ are formal variables and
$$T_{ij}(u)=\delta_{ij}u^p+\sum_{k=1}^p{t_{ij}^{(k)}u^{p-k}}\in Y_p(\gl_n)[u].$$
These relations are equivalent to the following equation

\begin{equation*}
\label{genYanequiv}
\left[t_{ij}^{(r)},t_{kl}^{(s)}\right]=\sum_{a=1}^{\min\{r,s\}}{\left(t_{kj}^{(a-1)}t_{il}^{(r+s-a)}-t_{kj}^{(r+s-a)}t_{il}^{(a-1)}\right)},
\end{equation*}
where $t_{ij}^{(0)}=\delta_{ij}$ and $t_{ij}^{(r)}=0$ for $r\geq p+1$.

Note that the level $1$ Yangian $Y_1(\gl_n)$ coincides with the universal enveloping algebra $U(\gl_n)$. Set
$$\deg(t_{ij}^{(k)})=k,\ \ \mbox{for}\ \ i,j=1,2,\ldots,n\ \ \mbox{and}\ \ k=1,2,\ldots, p.$$
This defines a natural filtration on Yangian $Y_p(\gl_n)$. The corresponding graded algebra will be denoted by $\overline{Y}_p(\gl_n)$ or $\gr(Y_p(\gl_n))$. Also, we have the following analogue of the Poincaré-Birkhoff-Witt Theorem for $Y_p(\gl_n)$.


\begin{proposition}[\textbf{PBW for Yangians}]
\label{PBWYan}
\noindent

The associated graded algebra $\overline{Y}_p(\gl_n)=\gr(Y_p(\gl_n))$ is a polynomial algebra in the variables $\overline{t}_{ij}^{(k)}$, $i,j=1,2,\ldots,n$ and $k=1,2,\ldots,p$.
\end{proposition} 


\begin{proof}
See \cite{CHE} or also Theorem ($2.1$) in \cite{Mo}. 

\end{proof}


We note that the Proposition \eqref{PBWYan} show that $Y_p(\gl_n)$ is a special filtered algebra. Now, consider the matrix $T(u)=\left(T_{ij}(u)\right)_{i,j=1}^n$ and the element in $Y_p(\gl_n)[u]$, called \textbf{quantum determinant} for $Y_p(\gl_n)$ and defined by

$$\qdet T(u)=\sum_{\sigma\in S_n}{\textup{sgn}(\sigma)T_{1\sigma(1)}(u)T_{2\sigma(2)}(u-1)\cdots T_{n\sigma(n)}(u-n+1)}.$$


\begin{example}

The \textit{quantum determinant} for $Y_p(\gl_2)$ is
\begin{align*}
\qdet T(u)&=T_{11}(u)T_{22}(u-1)-T_{21}(u)T_{12}(u-1)\\
    &=T_{11}(u-1)T_{22}(u)-T_{12}(u-1)T_{21}(u)\\
    &=T_{22}(u)T_{11}(u-1)-T_{12}(u)T_{21}(u-1)\\
    &=T_{22}(u-1)T_{11}(u)-T_{21}(u-1)T_{12}(u).
\end{align*}
\end{example}


Is not difficult to see that $\qdet T(u)$ is a monic polynomial in $u$ of degree $np$
$$\textup{qdet}\ T(u)=u^{np}+d_{n1}u^{np-1}+d_{n2}u^{np-2}+\cdots+d_{n\,np-1}u+d_{n\,np},\ \ \ d_{ni}\in Y_p(\gl_n).$$

\vspace{0.1cm}

\begin{proposition}
\label{centerZ(Yp(n))}
The center $Z(Y_p(\gl_n))$ of $Y_p(\gl_n)$ is generated by coefficients $d_{ni}$ of the powers $u^{np-i}$ of the quantum determinant $\textup{qdet}\ T(u)$,  i.e., 
$$Z(Y_p(\gl_n))=\left\langle d_{n1},d_{n2},\ldots,d_{n\,np}\right\rangle.$$ 
\end{proposition} 


\begin{proof}
See Theorem $(1)$ in \cite{CHE} or Corollary $(4.1)$ in \cite{Mo}.

\end{proof}


\vspace{0.2cm}

\noindent Now, consider the chain 
$$Y_p(\gl_1)\subset Y_p(\gl_2)\subset\cdots\subset Y_p(\gl_{n-1})\subset Y_p(\gl_n).$$
The \textbf{Gelfand-Tsetlin subalgebra} $\Gamma$ for Yangians $Y_p(\gl_n)$ is the subalgebra of $Y_p(\gl_n)$ generated by centers $Z(Y_p(\gl_i))$ of each subalgebra $Y_p(\gl_i)$ ($i=1,2,\ldots,n$). By Proposition 
\eqref{centerZ(Yp(n))}, the coefficients of the  quantum determinant 
$$\left\{d_{ij}\right\}_{\substack{
  i=1,2,\ldots,n \\
  j=1,2,\ldots,ip
 }}$$
generates $\Gamma$. Note that this algebra is commutative. 


\begin{remark}
\label{grdetT(u)}
For
$$F=\sum_{i}f_iu^i\in Y_p(\gl_n)[u],$$
we will use the following notation 
$$\overline{F}:=\sum_{i}\overline{f_i}u^i\in \overline{Y}_p(\gl_n)[u].$$ 
Also, denote 
$$X_{ij}^{(k)}:=\overline{t}_{ij}^{(k)},\ \ \ X_{ij}(u):=\overline{T}_{ij}(u),\ \ \ \mbox{and}\ \ \ X(u):=\left(X_{ij}(u)\right)_{i,j=1}^n.$$
Therefore, as $\ \overline{T_{ij}(u-\alpha)}=X_{ij}(u),\ \forall \alpha\in k,$ is not difficult to see
$$\gr\textup{qdet}\,T(u)=\det X(u).$$
\end{remark}


\vspace{0.5cm}

\noindent By PBW theorem for Yangians (Proposition \eqref{PBWYan}) each $\overline{d}_{ij}$ is a polynomial in the commutative variables $\overline{t}_{ij}^{(k)}$ with $i,j=1,2,\ldots,n$ and $k=1,2,\ldots,p$.


\begin{definition}
\label{GTsVar}
The \textbf{Gelfand-Tsetlin Variety} for $Y_p(\gl_n)$ 
is the algebraic variety  
$$\gts:=V\left(\left\{\overline{d_{ij}}\right\}_{\substack{
  i=1,2,\ldots,n \\
  j=1,2,\ldots,ip
 }}\right)\subset \k^{n^2p}.$$
\end{definition}


\begin{remark}
\label{Y1(n)gl1}
 By Corollary \eqref{KWvsGTs}, the Gelfand-Tsetlin variety $\gts$ for $Y_1(\gl_n)$ coincides with the Gelfand-Tsetlin variety for $\gl_n$. 
\end{remark}



\subsubsection{Some polynomials of the Gelfand-Tsetlin variety}

By Remark \eqref{grdetT(u)}
\begin{align*}
\textup{gr\,qdet}\,T(u) &=\det X(u)=\sum_{\sigma\in S_n}{\textup{sgn}(\sigma)X_{1\sigma(1)}(u)X_{2\sigma(2)}(u)\cdots X_{n\sigma(n)}(u)},\\
\textup{det}\, X(u)&=u^{np}+\overline{d}_{n1}u^{np-1}+\cdots+\overline{d}_{n\,np-1}u+\overline{d}_{n\,np},\ \ \ \overline{d}_{ni}\in \overline{Y}_p(\gl_n)[u]
\end{align*}
and 
$$X_{ij}(u)=\delta_{ij}u^p+\sum_{k=1}^p{X_{ij}^{(k)}u^{p-k}}\in \overline{Y}_p(\gl_n)[u],$$
with $X_{ij}^{(0)}=\delta_{ij}$ and $X_{ij}^{(r)}=0$ for $r\geq p+1$, we have  

$$\overline{d}_{ni}=\sum_{\sigma\in S_n}\sum_{t_1+t_2+\cdots+t_n=i}{\textup{sgn}(\sigma)X_{1\sigma(1)}^{(t_1)}X_{2\sigma(2)}^{(t_2)}\cdots X_{n\sigma(n)}^{(t_n)}}.$$


\begin{example} Some polynomials are 
\begin{align*}
\overline{d}_{n1}&=X_{11}^{(1)}+X_{22}^{(1)}+\cdots+X_{nn}^{(1)},\\
\overline{d}_{n2}&=\sum_{i=1}^{n}X_{ii}^{(2)}+\sum_{i=1}^{n-1}\sum_{j=i+1}^{n}\left(X_{ii}^{(1)}X_{jj}^{(1)}-X_{ij}^{(1)}X_{ji}^{(1)}\right).
\end{align*}

\end{example}

\begin{remark}
Another expression for the polynomials $\overline{d}_{ni}$ is the following: For any $i\in\left\{1,2,\ldots,np\right\}$ we consider the Young diagrams $\lambda=\left(\lambda_1,\lambda_2,\ldots,\lambda_r\right)$ with $r\leq n$ and $\lambda_{1}\leq p$, i.e., 
$\lambda_1,\lambda_2,\ldots,\lambda_r\in\Z_{>0}$ such that
  $$p\geq \lambda_1\geq \lambda_2\geq \ldots\geq \lambda_r>0,\ \ \ \lambda_1+\lambda_2+\ldots+\lambda_r=i\ \ \ \mbox{and}\ \ \ 1\leq r\leq n.$$
Denote the set of all this Young diagrams
$\lambda=\left(\lambda_1,\lambda_2,\ldots,\lambda_r\right)$ of $i$ with lenght $r$ and $\lambda_{1}\leq p$ by $\Omega_r(i)$. Hence, we can see that
$$\overline{d}_{ni}=\sum_{j=1}^{n}
\sum_{\substack{A\subset\left\{1,2,\ldots,n\right\} \\
  \left|A\right|=j\\
  \sigma\in S_{j}}}
\sum_{\substack{\lambda\in\Omega_j(i)\\
  \sigma'\in S_j}}\textup{sgn}(\sigma)X_{a_1a_{\sigma(1)}}^{(\lambda_{\sigma'(1)})}X_{a_2a_{\sigma(2)}}^{(\lambda_{\sigma'(2)})}\cdots X_{a_ja_{\sigma(j)}}^{(\lambda_{\sigma'(j)})},$$
where $A=\{a_1,a_2,\dots,a_j\}$ with $a_1<a_2<\cdots<a_j$, $\lambda=\left(\lambda_1,\lambda_2,\ldots,\lambda_r\right)$ and $S_j$ the permutations group of $\{1,2,\ldots,j\}$.
 
\end{remark}



\subsection{Gelfand-Tsetlin variety for Restricted Yangians}

We note that the Gel\-fand-Tsetlin variety $\gts$ for $Y_p(\gl_n)$ has polynomials
\begin{align*}
\overline{d}_{11}&=X_{11}^{(1)},\\
\overline{d}_{21}&=X_{11}^{(1)}+X_{22}^{(1)},\\
\overline{d}_{31}&=X_{11}^{(1)}+X_{22}^{(1)}+X_{33}^{(1)},\\
		 &\vdots\\
\overline{d}_{n1}&=X_{11}^{(1)}+X_{22}^{(1)}+\cdots+X_{nn}^{(1)},
\end{align*}
but, clearly
  $$V\left(\overline{d}_{11},\overline{d}_{21},\dots,\overline{d}_{n1}\right)=V\left(X_{11}^{(1)},X_{22}^{(1)},\dots,X_{nn}^{(1)}\right).$$
Therefore, we would like to improve the polynomials that determined the Gelfand-Tsetlin variety. With that aim, we will understand the polynomials for the case $n=3$ and we will can to note that for cases $n>3$ the combinatorics is not simple.


\begin{lemma}
\label{GTsPolyn10}
The polynomials that determine the Gelfand-Tsetlin variety $\gts$ for $Y_p(\gl_3)$ are:
\begin{align*}
\overline{d}_{1i}    &=X_{11}^{(i)},\ \ \ i=1,2,\dots,p,\\
\overline{d}_{21}    &=X_{11}^{(1)}+X_{22}^{(1)},\\
\overline{d}_{2i}    &=X_{11}^{(i)}+X_{22}^{(i)}+\sum_{t=1}^{i-1}\left(X_{11}^{(t)}X_{22}^{(i-t)}-X_{12}^{(t)}X_{21}^{(i-t)}\right),\ \ \ i=2,3,\ldots,p,\\
\overline{d}_{2\,p+i}&=\sum_{t=i}^{p}\left(X_{11}^{(t)}X_{22}^{(p+i-t)}-X_{12}^{(t)}X_{21}^{(p+i-t)}\right), \ \ \ i=1,2,3,\ldots,p,\\
\overline{d}_{31}    &=X_{11}^{(1)}+X_{22}^{(1)}+X_{33}^{(1)},\\
\overline{d}_{32}    &=X_{11}^{(2)}+X_{22}^{(2)}+X_{33}^{(2)}+X_{11}^{(1)}X_{22}^{(1)}+X_{11}^{(1)}X_{33}^{(1)}+X_{22}^{(1)}X_{33}^{(1)}-X_{23}^{(1)}X_{32}^{(1)}-\\
					 &\ \ \ -X_{12}^{(1)}X_{21}^{(1)}-X_{13}^{(1)}X_{31}^{(1)},\\
\overline{d}_{3i}    &=X_{33}^{(i)}+X_{22}^{(i)}+X_{11}^{(i)}+ \sum_{s=1}^{i-1}\left(X_{22}^{(s)}X_{33}^{(i-s)}-X_{23}^{(s)}X_{32}^{(i-s)}\right)+\\
                     &\ \ \ +\sum_{s=1}^{i-1}\left(X_{11}^{(s)}X_{33}^{(i-s)}-X_{13}^{(s)}X_{31}^{(i-s)}+X_{11}^{(s)}X_{22}^{(i-s)}-X_{12}^{(s)}X_{21}^{(i-s)}\right)+\\
                     &\ \ \ +\sum_{r=1}^{i-2}\sum_{s=1}^{i-r-1}\biggl\{X_{11}^{(r)}X_{22}^{(s)}X_{33}^{(i-r-s)}-X_{11}^{(r)}X_{23}^{(s)}X_{32}^{(i-r-s)}-X_{12}^{(r)}X_{21}^{(s)}X_{33}^{(i-r-s)}-\\
                     &\ \ \ -X_{13}^{(r)}X_{22}^{(s)}X_{31}^{(i-r-s)}+X_{12}^{(r)}X_{23}^{(s)}X_{31}^{(i-r-s)}+X_{13}^{(r)}X_{21}^{(s)}X_{32}^{(i-r-s)}\biggr\},\\
                     & \ \ \ \ \mbox{for}\ \ i=3,4,\ldots,p,
\end{align*}
\begin{align*}
\overline{d}_{3\,p+1}&=\sum_{s=1}^p \left(X_{22}^{(s)}X_{33}^{(p+1-s)}-X_{23}^{(s)}X_{32}^{(p+1-s)}+X_{11}^{(s)}X_{33}^{(p+1-s)}-X_{13}^{(s)}X_{31}^{(p+1-s)}\right)+\\
                     &\ \ \ +\sum_{s=1}^p\left(X_{11}^{(s)}X_{22}^{(p+1-s)}-X_{12}^{(s)}X_{21}^{(p+1-s)}\right)+\\
                     &\ \ \ +\sum_{r=1}^{p-1}\sum_{s=1}^{p-r}\biggl\{X_{11}^{(r)}X_{22}^{(s)}X_{33}^{(p+1-r-s)}-X_{11}^{(r)}X_{23}^{(s)}X_{32}^{(p+1-r-s)}-\\
                     &\ \ \ -X_{12}^{(r)}X_{21}^{(s)}X_{33}^{(p+1-r-s)}-X_{13}^{(r)}X_{22}^{(s)}X_{31}^{(p+1-r-s)}+X_{12}^{(r)}X_{23}^{(s)}X_{31}^{(p+1-r-s)}+\\
                     &\ \ \ +X_{13}^{(r)}X_{21}^{(s)}X_{32}^{(p+1-r-s)}\biggr\},\\
\overline{d}_{3\,p+i}&= \sum_{s=i}^p\left(X_{22}^{(s)}X_{33}^{(p+i-s)}-X_{23}^{(s)}X_{32}^{(p+i-s)}+X_{11}^{(s)}X_{33}^{(p+i-s)}\right)+\\
                     &\ \ \ +\sum_{s=i}^p\left(-X_{13}^{(s)}X_{31}^{(p+i-s)}+X_{11}^{(s)}X_{22}^{(p+i-s)}-X_{12}^{(s)}X_{21}^{(p+i-s)}\right)+\\
                     &\ \ \ +\sum_{r=1}^{i-1}\sum_{s=i-r}^p\biggl\{X_{11}^{(r)}X_{22}^{(s)}X_{33}^{(p+i-r-s)}-X_{11}^{(r)}X_{23}^{(s)}X_{32}^{(p+i-r-s)}-\\
                     &\ \ \ -X_{12}^{(r)}X_{21}^{(s)}X_{33}^{(p+i-r-s)}-X_{13}^{(r)}X_{22}^{(s)}X_{31}^{(p+i-r-s)}+X_{12}^{(r)}X_{23}^{(s)}X_{31}^{(p+i-r-s)}+\\
                     &\ \ \ +X_{13}^{(r)}X_{21}^{(s)}X_{32}^{(p+i-r-s)}\biggr\}+\\
                     &\ \ \ +\sum_{r=i}^{p}\sum_{s=1}^{p+i-r-1}\biggl\{X_{11}^{(r)}X_{22}^{(s)}X_{33}^{(p+i-r-s)}-X_{11}^{(r)}X_{23}^{(s)}X_{32}^{(p+i-r-s)}-\\
                     &\ \ \ -X_{12}^{(r)}X_{21}^{(s)}X_{33}^{(p+i-r-s)}-X_{13}^{(r)}X_{22}^{(s)}X_{31}^{(p+i-r-s)}+X_{12}^{(r)}X_{23}^{(s)}X_{31}^{(p+i-r-s)}+\\
                     &\ \ \ +X_{13}^{(r)}X_{21}^{(s)}X_{32}^{(p+i-r-s)}\biggr\},\ \ i=2,3,\ldots,p,\\
\overline{d}_{3\,2p+i}&=\sum_{r=i}^{p}\sum_{s=p+i-r}^{p}\biggl\{X_{11}^{(r)}X_{22}^{(s)}X_{33}^{(2p+i-r-s)}-X_{11}^{(r)}X_{23}^{(s)}X_{32}^{(2p+i-r-s)}-\\
					  &\ \ \ -X_{12}^{(r)}X_{21}^{(s)}X_{33}^{(2p+i-r-s)}-X_{13}^{(r)}X_{22}^{(s)}X_{31}^{(2p+i-r-s)}+\\
					  &\ \ \ +X_{12}^{(r)}X_{23}^{(s)}X_{31}^{(2p+i-r-s)}+X_{13}^{(r)}X_{21}^{(s)}X_{32}^{(2p+i-r-s)}\biggr\},\ \ i=1,2,\ldots,p.
\end{align*}
\end{lemma}


\begin{proof}
 It is easy to see for the polynomials $\overline{d}_{1i}$ and $\overline{d}_{2i}$ for any $i=1,2,\dots,p$. Now, for $i=1,2,3,\ldots,p$
\begin{align*}
\overline{d}_{2\,p+i}&=\sum_{\substack{t_1+t_2=p+i\\
			i\leq t_1,t_2\leq p}}X_{11}^{(t_1)}X_{22}^{(t_2)} -\sum_{\substack{t_1+t_2=p+i\\
			i\leq t_1,t_2\leq p}}X_{12}^{(t_1)}X_{21}^{(t_2)}\ \ (\mbox{because}\ \ X_{ij}^{(r)}=0,\ \forall r>p)\\
		     &=\sum_{t=i}^{p}X_{11}^{(t)}X_{22}^{(p+i-t)}-\sum_{t=i}^{p}X_{12}^{(t)}X_{21}^{(p+i-t)}.
\end{align*}
Since that for $i=1,2,\ldots,p$
\begin{align*}
\overline{d}_{3i}&=\sum_{t_1+t_2+t_3=i} \biggl\{X_{11}^{(t_1)}X_{22}^{(t_2)}X_{33}^{(t_3)}-X_{11}^{(t_1)}X_{23}^{(t_2)}X_{32}^{(t_3)}-X_{12}^{(t_1)}X_{21}^{(t_2)}X_{33}^{(t_3)}-\\
		 &\ \ \ \ -X_{13}^{(t_1)}X_{22}^{(t_2)}X_{31}^{(t_3)}+X_{12}^{(t_1)}X_{23}^{(t_2)}X_{31}^{(t_3)}+X_{13}^{(t_1)}X_{21}^{(t_2)}X_{32}^{(t_3)}\biggr\},
\end{align*}
then, by proposition (A.2) in \cite{gabm},
\begin{align*}
\overline{d}_{31}&=X_{11}^{(1)}+X_{22}^{(1)}+X_{33}^{(1)},\\
\overline{d}_{32}&=X_{11}^{(2)}+X_{22}^{(2)}+X_{33}^{(2)}+X_{11}^{(1)}X_{22}^{(1)}+X_{11}^{(1)}X_{33}^{(1)}+X_{22}^{(1)}X_{33}^{(1)}-X_{23}^{(1)}X_{32}^{(1)}-\\
				 &\ \ \ -X_{12}^{(1)}X_{21}^{(1)}-X_{13}^{(1)}X_{31}^{(1)},\\
\overline{d}_{3i}&=X_{33}^{(i)}+X_{22}^{(i)}+X_{11}^{(i)}+ \sum_{s=1}^{i-1}\left(X_{22}^{(s)}X_{33}^{(i-s)}-X_{23}^{(s)}X_{32}^{(i-s)}\right)+\\
                 &\ \ \ +\sum_{s=1}^{i-1}\left(X_{11}^{(s)}X_{33}^{(i-s)}-X_{13}^{(s)}X_{31}^{(i-s)}+X_{11}^{(s)}X_{22}^{(i-s)}-X_{12}^{(s)}X_{21}^{(i-s)}\right)+\\
                 &\ \ \ +\sum_{r=1}^{i-2}\sum_{s=1}^{i-r-1}\biggl\{X_{11}^{(r)}X_{22}^{(s)}X_{33}^{(i-r-s)}-X_{11}^{(r)}X_{23}^{(s)}X_{32}^{(i-r-s)}-X_{12}^{(r)}X_{21}^{(s)}X_{33}^{(i-r-s)}-\\
                 &\ \ \ -X_{13}^{(r)}X_{22}^{(s)}X_{31}^{(i-r-s)}+X_{12}^{(r)}X_{23}^{(s)}X_{31}^{(i-r-s)}+X_{13}^{(r)}X_{21}^{(s)}X_{32}^{(i-r-s)}\biggr\},\\
                 &\ \ \ \ \ i=3,4,\ldots,p.
\end{align*}
Similarly, since that $i=1,2,\ldots,p$ and $X_{ij}^{(r)}=0,\ \ \forall r>p$
\begin{align*}
\overline{d}_{3\,p+i}&=\sum_{\substack{t_1+t_2+t_3=p+i\\
			0\leq t_1,t_2,t_3\leq p}}\biggl\{X_{11}^{(t_1)}X_{22}^{(t_2)}X_{33}^{(t_3)}-X_{11}^{(t_1)}X_{23}^{(t_2)}X_{32}^{(t_3)}-X_{12}^{(t_1)}X_{21}^{(t_2)}X_{33}^{(t_3)}-\\
		     &\ \ \ \ -X_{13}^{(t_1)}X_{22}^{(t_2)}X_{31}^{(t_3)}+X_{12}^{(t_1)}X_{23}^{(t_2)}X_{31}^{(t_3)}+X_{13}^{(t_1)}X_{21}^{(t_2)}X_{32}^{(t_3)}\biggr\},
\end{align*}
then, of the proposition (A.3) in \cite{gabm} follow that
\begin{align*}
\overline{d}_{3\,p+1}&=\sum_{\substack{t_1+t_2+t_3=p+1\\
  0\leq t_1,t_2,t_3\leq p}}\biggl\{X_{11}^{(t_1)}X_{22}^{(t_2)}X_{33}^{(t_3)}-X_{11}^{(t_1)}X_{23}^{(t_2)}X_{32}^{(t_3)}-X_{12}^{(t_1)}X_{21}^{(t_2)}X_{33}^{(t_3)}-\\
                     &\ \ \ -X_{13}^{(t_1)}X_{22}^{(t_2)}X_{31}^{(t_3)}+X_{12}^{(t_1)}X_{23}^{(t_2)}X_{31}^{(t_3)}+X_{13}^{(t_1)}X_{21}^{(t_2)}X_{32}^{(t_3)}\biggr\}\\
                     &=\sum_{s=1}^p \left(X_{22}^{(s)}X_{33}^{(p+1-s)}-X_{23}^{(s)}X_{32}^{(p+1-s)}+X_{11}^{(s)}X_{33}^{(p+1-s)}-X_{13}^{(s)}X_{31}^{(p+1-s)}\right)+\\
                     &\ \ \ +\sum_{s=1}^p\left(X_{11}^{(s)}X_{22}^{(p+1-s)}-X_{12}^{(s)}X_{21}^{(p+1-s)}\right)+\\
                     &\ \ \ +\sum_{r=1}^{p-1}\sum_{s=1}^{p-r}\biggl\{X_{11}^{(r)}X_{22}^{(s)}X_{33}^{(p+1-r-s)}-X_{11}^{(r)}X_{23}^{(s)}X_{32}^{(p+1-r-s)}-\\
                     &\ \ \ -X_{12}^{(r)}X_{21}^{(s)}X_{33}^{(p+1-r-s)}-X_{13}^{(r)}X_{22}^{(s)}X_{31}^{(p+1-r-s)}+\\
                     &\ \ \ +X_{12}^{(r)}X_{23}^{(s)}X_{31}^{(p+1-r-s)}+X_{13}^{(r)}X_{21}^{(s)}X_{32}^{(p+1-r-s)}\biggr\},\\ 
\overline{d}_{3\,p+i}&= \sum_{s=i}^p\left(X_{22}^{(s)}X_{33}^{(p+i-s)}-X_{23}^{(s)}X_{32}^{(p+i-s)}+X_{11}^{(s)}X_{33}^{(p+i-s)}\right)+\\
                     &\ \ \ +\sum_{s=i}^p\left(-X_{13}^{(s)}X_{31}^{(p+i-s)}+X_{11}^{(s)}X_{22}^{(p+i-s)}-X_{12}^{(s)}X_{21}^{(p+i-s)}\right)+\\
                     &\ \ \ +\sum_{r=1}^{i-1}\sum_{s=i-r}^p\biggl\{X_{11}^{(r)}X_{22}^{(s)}X_{33}^{(p+i-r-s)}-X_{11}^{(r)}X_{23}^{(s)}X_{32}^{(p+i-r-s)}-\\
                     &\ \ \ -X_{12}^{(r)}X_{21}^{(s)}X_{33}^{(p+i-r-s)}-X_{13}^{(r)}X_{22}^{(s)}X_{31}^{(p+i-r-s)}+\\
                     &\ \ \ +X_{12}^{(r)}X_{23}^{(s)}X_{31}^{(p+i-r-s)}+X_{13}^{(r)}X_{21}^{(s)}X_{32}^{(p+i-r-s)}\biggr\}+\\
                     &\ \ \ +\sum_{r=i}^{p}\sum_{s=1}^{p+i-r-1}\biggl\{X_{11}^{(r)}X_{22}^{(s)}X_{33}^{(p+i-r-s)}-X_{11}^{(r)}X_{23}^{(s)}X_{32}^{(p+i-r-s)}-\\
                     &\ \ \ -X_{12}^{(r)}X_{21}^{(s)}X_{33}^{(p+i-r-s)}-X_{13}^{(r)}X_{22}^{(s)}X_{31}^{(p+i-r-s)}+\\
                     &\ \ \ +X_{12}^{(r)}X_{23}^{(s)}X_{31}^{(p+i-r-s)}+X_{13}^{(r)}X_{21}^{(s)}X_{32}^{(p+i-r-s)}\biggr\},\ \ \ i=2,3,\ldots,p.
\end{align*}
Finally, as $X_{ij}^{(r)}=0,\ \ \forall r>p$ and by proposition (A.4) in \cite{gabm}, for each $i=1,2,\ldots,p$  
\begin{align*}                     
\overline{d}_{3\,2p+i}&=\sum_{\substack{t_1+t_2+t_3=2p+i\\
  0\leq t_1,t_2,t_3\leq p}}\biggl\{X_{11}^{(t_1)}X_{22}^{(t_2)}X_{33}^{(t_3)}-X_{11}^{(t_1)}X_{23}^{(t_2)}X_{32}^{(t_3)}-X_{12}^{(t_1)}X_{21}^{(t_2)}X_{33}^{(t_3)}-\\
					&\ \ \ -X_{13}^{(t_1)}X_{22}^{(t_2)}X_{31}^{(t_3)}+X_{12}^{(t_1)}X_{23}^{(t_2)}X_{31}^{(t_3)}+X_{13}^{(t_1)}X_{21}^{(t_2)}X_{32}^{(t_3)}\biggr\}\\
					&=\sum_{r=i}^{p}\sum_{s=p+i-r}^{p}\biggl\{X_{11}^{(r)}X_{22}^{(s)}X_{33}^{(2p+i-r-s)}-X_{11}^{(r)}X_{23}^{(s)}X_{32}^{(2p+i-r-s)}-\\
					&\ \ \ -X_{12}^{(r)}X_{21}^{(s)}X_{33}^{(2p+i-r-s)}-X_{13}^{(r)}X_{22}^{(s)}X_{31}^{(2p+i-r-s)}+\\
					&\ \ \ +X_{12}^{(r)}X_{23}^{(s)}X_{31}^{(2p+i-r-s)}+X_{13}^{(r)}X_{21}^{(s)}X_{32}^{(2p+i-r-s)}\biggr\}.
\end{align*}

\end{proof}


\vspace{0.1cm}

As consequence of the previous lemma, we can show the Gelfand-Tsetlin variety for $Y_p(\gl_3)$ with other polynomials a little less complicated. For more details in the computations see proposition ($5.9$) pages $56-63$ in \cite{gabm}.


\begin{proposition}
\label{equivGTs}

$$V\left(\left\{\overline{d}_{ij}\right\}_{\substack{
  i=1,2,3\ \\
  j=1,2,\ldots,ip
 }}\right)=V\left(\left\{p_{ij}\right\}_{\substack{
  i=1,2,3\ \\
  j=1,2,\ldots,ip
 }}\right)\subset \k^{9p},$$
 where
 \begin{align*}
p_{1i}      &=X_{11}^{(i)},\ \ \ i=1,2,\ldots,p,\\
p_{21}      &=X_{22}^{(1)},\\
p_{2i}      &=X_{22}^{(i)}-\sum_{t=1}^{i-1}X_{12}^{(t)}X_{21}^{(i-t)},\ \ \ i=2,3,\ldots,p,\\
p_{2\,p+i}  &=\sum_{t=i}^{p}X_{12}^{(t)}X_{21}^{(p+i-t)},\ \ \ i=1,2,\ldots,p,\\
p_{31}      &=X_{33}^{(1)},\\
p_{32}      &= X_{33}^{(2)}-X_{23}^{(1)}X_{32}^{(1)}-X_{13}^{(1)}X_{31}^{(1)},\\
p_{3i}      &= -\sum_{r=1}^{i-2}\sum_{s=2}^{i-r-1}X_{13}^{(r)}X_{22}^{(s)}X_{31}^{(i-r-s)}+X_{33}^{(i)}-\sum_{s=1}^{i-1}\left(X_{13}^{(s)}X_{31}^{(i-s)}+X_{23}^{(s)}X_{32}^{(i-s)}\right)+\\
            &\ \ \ +\sum_{r=1}^{i-2}\sum_{s=1}^{i-r-1}\biggl\{X_{12}^{(r)}X_{23}^{(s)}X_{31}^{(i-r-s)}+X_{13}^{(r)}X_{21}^{(s)}X_{32}^{(i-r-s)}\biggr\},\ \ \ i=3,4,\ldots,p, \\
p_{3\,p+1}  &=\sum_{r=1}^{p-1}\sum_{s=1}^{p-r}\biggl\{ X_{12}^{(r)}X_{23}^{(s)}X_{31}^{(p+1-r-s)}+X_{13}^{(r)}X_{21}^{(s)}X_{32}^{(p+1-r-s)}\biggr\}-\\
            &\ \ \ -\sum_{s=1}^p\left(X_{13}^{(s)}X_{31}^{(p+1-s)}+ X_{23}^{(s)}X_{32}^{(p+1-s)}\right)-\sum_{r=1}^{p-1}\sum_{s=2}^{p-r}X_{13}^{(r)}X_{22}^{(s)}X_{31}^{(p+1-r-s)},\\
p_{3\,p+i}  &=-\sum_{s=i}^p\left(X_{23}^{(s)}X_{32}^{(p+i-s)}-X_{13}^{(s)}X_{31}^{(p+i-s)}\right)+\\
            &\ \ \ +\sum_{r=1}^{i-1}\sum_{s=i-r}^p\biggl\{-X_{13}^{(r)}X_{22}^{(s)}X_{31}^{(p+i-r-s)}+X_{12}^{(r)}X_{23}^{(s)}X_{31}^{(p+i-r-s)}+\\
            &\ \ \ +X_{13}^{(r)}X_{21}^{(s)}X_{32}^{(p+i-r-s)}\biggr\}+\\
            &\ \ \ +\sum_{r=i}^{p}\sum_{s=1}^{p+i-r-1}\biggl\{-X_{13}^{(r)}X_{22}^{(s)}X_{31}^{(p+i-r-s)}+X_{12}^{(r)}X_{23}^{(s)}X_{31}^{(p+i-r-s)}+\\
            &\ \ \ +X_{13}^{(r)}X_{21}^{(s)}X_{32}^{(p+i-r-s)}\biggr\},\ \ \ i=2,3,\ldots,p,\\      
p_{3\, 2p+i}&=\sum_{r=i}^{p}\sum_{s=p+i-r}^{p} \biggl\{-X_{13}^{(r)}X_{22}^{(s)}X_{31}^{(2p+i-r-s)}+X_{12}^{(r)}X_{23}^{(s)}X_{31}^{(2p+i-r-s)}+\\
			&\ \ \ +X_{13}^{(r)}X_{21}^{(s)}X_{32}^{(2p+i-r-s)}\biggr\},\ \ \ i=1,2,\ldots,p.
\end{align*}
\end{proposition}


\begin{proof} 
Clearly $\overline{d}_{1i}=p_{1i}$ ($i=1,2,\ldots,p$) and $\overline{d}_{21}=p_{11}+p_{21}$. Now, for $i=2,3,\ldots,p$
\begin{align*}
\overline{d}_{2i}    &= p_{1i}+\sum_{t=1}^{i-1}p_{1t}X_{22}^{(i-t)}+p_{2i}.
\end{align*}
Similarly, for each $i=1,2,\ldots,p$
\begin{align*}
\overline{d}_{2\,p+i}&= \sum_{t=i}^{p}p_{1t}X_{22}^{(p+i-t)}-p_{2\,p+i}.
\end{align*}
Continuing with the same argument $\overline{d}_{31}=p_{11}+p_{21}+p_{31}$ and
\begin{align*}
\overline{d}_{32}&= p_{12}+p_{22}+p_{11}X_{22}^{(1)}+p_{11}X_{33}^{(1)}+p_{21}X_{33}^{(1)}+p_{32}.
\end{align*}
For each $i=3,4,\ldots,p$
\begin{align*}
\overline{d}_{3i}    &= p_{1i}+\sum_{s=1}^{i-1}p_{1s}\left(X_{33}^{(i-s)}+X_{22}^{(i-s)}\right)+\\
&\ \ \ +\sum_{r=1}^{i-2}\sum_{s=1}^{i-r-1}p_{1r}\left(X_{22}^{(s)}X_{33}^{(i-r-s)}-X_{23}^{(s)}X_{32}^{(i-r-s)}\right)+\\
                     &\ \ \ +p_{2i}+p_{21}X_{33}^{(i-1)}+ \sum_{r=1}^{i-2}p_{2\,i-r}X_{33}^{(r)}+p_{3i}.
\end{align*}
Analogously
\begin{align*}
\overline{d}_{3\,p+1}&= \sum_{s=1}^pp_{1s}\left(X_{33}^{(p+1-s)}+X_{22}^{(p+1-s)}\right)+\\
					 &\ \ \ +\sum_{r=1}^{p-1}\sum_{s=1}^{p-r}p_{1r}\left(X_{22}^{(s)}X_{33}^{(p+1-r-s)}-X_{23}^{(s)}X_{32}^{(p+1-r-s)}\right)+\\
                     &\ \ \ +p_{21}X_{33}^{(p)}+\sum_{r=1}^{p-1}X_{33}^{(r)}p_{2\,p+1-r}-p_{2\,p+1}+p_{3\,p+1}.
\end{align*}                     
For all $i=2,3,\ldots,p$
\begin{align*}
\overline{d}_{3\,p+i}&= \sum_{s=i}^pp_{1s}\left(X_{22}^{(p+i-s)}+X_{33}^{(p+i-s)}\right)+\\
					 &\ \ \ +\sum_{r=1}^{i-1}\sum_{s=i-r}^pp_{1r}\left(X_{22}^{(s)}X_{33}^{(p+i-r-s)}-X_{23}^{(s)}X_{32}^{(p+i-r-s)}\right)+\\
                     &\ \ \ +\sum_{r=i}^{p}\sum_{s=1}^{p+i-r-1}p_{1r}\left(X_{22}^{(s)}X_{33}^{(p+i-r-s)}-X_{23}^{(s)}X_{32}^{(p+i-r-s)}\right)-p_{2i}+\\
                     &\ \ \ +\sum_{r=i}^pX_{33}^{(r)}p_{2\,p+i-r}-\sum_{r=1}^{i-1}X_{33}^{(r)}p_{2\,p+i-r}+p_{3\,p+i}.
\end{align*}
Finally, for each $i=1,2,\ldots,p$
\begin{align*}
\overline{d}_{3\,2p+i}&= \sum_{r=i}^{p}\sum_{s=p+i-r}^{p}p_{1r}\left(X_{22}^{(s)}X_{33}^{(2p+i-r-s)}-X_{23}^{(s)}X_{32}^{(2p+i-r-s)}\right)-\\
&\ \ \ -\sum_{r=i}^{p}X_{33}^{(r)}p_{2\,p+i-r}+p_{3\,2p+i}.
\end{align*}

\end{proof}


\begin{corollary}
\label{equivGTsYp(n)}
The Gelfand-Tsetlin variety $\gts$ for $Y_p(\gl_n)$ have the form:
$$\gts=V\left(\left\{p_{ij}\right\}_{\substack{i=1,2,\ldots,n\\ j=1,2,\ldots,ip
 }}\right)\subset \k^{9p},$$
 where
 $$p_{ij}=\overline{d}_{ij},\ \ \mbox{for}\ \ 4\leq i\leq n\ \ \mbox{and}\ \ 1\leq j\leq ip.$$
\end{corollary}


\begin{proof}
 Definition \eqref{GTsVar} and by Proposition \eqref{equivGTs}.
 
\end{proof}


\subsection{Equidimensionality for some Gelfand-Tsetlin varieties}

By Corollary \eqref{KWvsGTs} follow that the Gelfand-Tsetlin variety $\gts$ for $Y_1(\gl_n)$ coincides with the Gelfand-Tsetlin variety for $\gl_n$ and by Ovsienko's Theorem \eqref{tmaOvs} we have that this variety is equidimensional of dimension 
  $$\dim\gts=\frac{n(n-1)}{2}.$$
Ovsienko's Theorem \eqref{tmaOvs} states that all irreducible components have the same dimension, but does not say anything about its decomposition in irreducible components and the following result we will show such decomposition. 


\begin{proposition}
\label{Y1(3)}
The Gelfand-Tsetlin variety $\gts$ for $Y_1(\gl_3)$ is equidimensional of dimension $\dim\gts=3$ and its decomposition in irreducible components is
$$\gts=\bigcup_{i=1}^7C_i$$ 
where
\begin{align*}
C_1&= V\left(X_{11}^{(1)},X_{22}^{(1)},X_{12}^{(1)},X_{33}^{(1)},X_{23}^{(1)},X_{13}^{(1)}\right),\\
C_2&= V\left(X_{11}^{(1)},X_{22}^{(1)},X_{12}^{(1)},X_{33}^{(1)},X_{32}^{(1)},X_{13}^{(1)}\right),\\
C_3&= V\left(X_{11}^{(1)},X_{22}^{(1)},X_{12}^{(1)},X_{33}^{(1)},X_{13}^{(1)}X_{31}^{(1)}+X_{23}^{(1)}X_{32}^{(1)},X_{21}^{(1)}\right),\\
C_4&= V\left(X_{11}^{(1)},X_{22}^{(1)},X_{12}^{(1)},X_{33}^{(1)},X_{31}^{(1)},X_{32}^{(1)}\right),\\
C_5&= V\left(X_{11}^{(1)},X_{22}^{(1)},X_{21}^{(1)},X_{33}^{(1)},X_{13}^{(1)},X_{23}^{(1)}\right),\\
C_6 & = V\left(X_{11}^{(1)},X_{22}^{(1)},X_{21}^{(1)},X_{33}^{(1)},X_{31}^{(1)},X_{23}^{(1)}\right),\\
C_7 & = V\left(X_{11}^{(1)},X_{22}^{(1)},X_{21}^{(1)},X_{33}^{(1)},X_{32}^{(1)},X_{31}^{(1)}\right).
\end{align*}
\end{proposition}


\begin{proof}
Follow of the Proposition \eqref{equivGTs} or by Corollary \eqref{equivGTsYp(n)} that
$$\gts=V\left(p_{11},p_{21},p_{22},p_{31},p_{32},p_{33}\right)\subset\k^{9},$$
where
$$p_{11}=X_{11}^{(1)},\ \ \ \
p_{21}=X_{22}^{(1)},\ \ \ \
p_{22}=X_{12}^{(1)}X_{21}^{(1)},\ \ \ \
p_{31}=X_{33}^{(1)},$$
$$p_{32}=X_{13}^{(1)}X_{31}^{(1)}+ X_{23}^{(1)}X_{32}^{(1)}\ \ \ \ \mbox{and}\ \ \ \ 
p_{33}=X_{12}^{(1)}X_{23}^{(1)}X_{31}^{(1)}+X_{13}^{(1)}X_{21}^{(1)}X_{32}^{(1)},$$ 
hence
\begin{align*}
\gts&=V\left(X_{11}^{(1)},X_{22}^{(1)},X_{12}^{(1)},X_{33}^{(1)},X_{13}^{(1)}X_{31}^{(1)}+X_{23}^{(1)}X_{32}^{(1)},X_{13}^{(1)}X_{21}^{(1)}X_{32}^{(1)}\right)\cup\\
    &\ \ \ \cup V\left(X_{11}^{(1)},X_{22}^{(1)},X_{21}^{(1)},X_{33}^{(1)},X_{13}^{(1)}X_{31}^{(1)}+X_{23}^{(1)}X_{32}^{(1)},X_{12}^{(1)}X_{23}^{(1)}X_{31}^{(1)}\right)\\
    &= V\left(X_{11}^{(1)},X_{22}^{(1)},X_{12}^{(1)},X_{33}^{(1)},X_{23}^{(1)}X_{32}^{(1)},X_{13}^{(1)}\right)\cup \\
    &\ \ \ \cup V\left(X_{11}^{(1)},X_{22}^{(1)},X_{12}^{(1)},X_{33}^{(1)},X_{13}^{(1)}X_{31}^{(1)}+X_{23}^{(1)}X_{32}^{(1)},X_{21}^{(1)}\right)\cup\\
    &\ \ \ \cup V\left(X_{11}^{(1)},X_{22}^{(1)},X_{12}^{(1)},X_{33}^{(1)},X_{13}^{(1)}X_{31}^{(1)},X_{32}^{(1)}\right)\cup \\
    &\ \ \ \cup V\left(X_{11}^{(1)},X_{22}^{(1)},X_{21}^{(1)},X_{33}^{(1)},X_{13}^{(1)}X_{31}^{(1)},X_{23}^{(1)}\right)\cup\\
    &\ \ \ \cup V\left(X_{11}^{(1)},X_{22}^{(1)},X_{21}^{(1)},X_{33}^{(1)},X_{23}^{(1)}X_{32}^{(1)},X_{31}^{(1)}\right)\\
    &= V\left(X_{11}^{(1)},X_{22}^{(1)},X_{12}^{(1)},X_{33}^{(1)},X_{23}^{(1)},X_{13}^{(1)}\right)\cup\\
    &\ \ \ \cup  V\left(X_{11}^{(1)},X_{22}^{(1)},X_{12}^{(1)},X_{33}^{(1)},X_{32}^{(1)},X_{13}^{(1)}\right)\cup\\
    &\ \ \ \cup V\left(X_{11}^{(1)},X_{22}^{(1)},X_{12}^{(1)},X_{33}^{(1)},X_{13}^{(1)}X_{31}^{(1)}+X_{23}^{(1)}X_{32}^{(1)},X_{21}^{(1)}\right)\cup\\
    &\ \ \ \cup  V\left(X_{11}^{(1)},X_{22}^{(1)},X_{12}^{(1)},X_{33}^{(1)},X_{31}^{(1)},X_{32}^{(1)}\right)\cup\\
    &\ \ \ \cup  V\left(X_{11}^{(1)},X_{22}^{(1)},X_{21}^{(1)},X_{33}^{(1)},X_{13}^{(1)},X_{23}^{(1)}\right)\cup\\
    &\ \ \ \cup  V\left(X_{11}^{(1)},X_{22}^{(1)},X_{21}^{(1)},X_{33}^{(1)},X_{31}^{(1)},X_{23}^{(1)}\right)\cup \\
    &\ \ \ \cup  V\left(X_{11}^{(1)},X_{22}^{(1)},X_{21}^{(1)},X_{33}^{(1)},X_{32}^{(1)},X_{31}^{(1)}\right).
\end{align*}
Therefore, $\gts$ for $Y_1(\gl_3)$ has $7$ irreducibles components and each component has dimension $3$.

\end{proof}


\vspace{0.2cm}

With the aim of to show the difficulty of the equidimensionality of $\gts$ calculating its decomposition, the following proposition guarantee the equidimensionality of the Gelfand-Tsetlin variety $\gts$ for $Y_2(\gl_3)$ exhibiting its decomposition in irreducible components. All the details of such decomposition are in \cite{gabm} (Appendix (B) pages $97-123$), basically we used the same argument of the previous proposition.


\begin{proposition}
\label{Y2(3)}
The Gelfand-Tsetlin variety $\gts$ for $Y_2(\gl_3)$ is equidimensional with $\dim\gts=6.$
\end{proposition}


\begin{proof}
By proposition \eqref{equivGTs} or corollary \eqref{equivGTsYp(n)}, we have that
$$\gts=V\left(\left\{p_{ij}\right\}_{\substack{
  i=1,2,3\ \\
  j=1,2,\ldots,2i
 }}\right)\subset\k^{18},$$
where 
 \begin{align*}
p_{11}&=X_{11}^{(1)},\\
p_{12}&=X_{11}^{(2)},\\
p_{21}&=X_{22}^{(1)},\\
p_{22}&=X_{22}^{(2)}-X_{12}^{(1)}X_{21}^{(1)},\\
p_{23}&=X_{12}^{(1)}X_{21}^{(2)}+X_{12}^{(2)}X_{21}^{(1)},\\
p_{24}&=X_{12}^{(2)}X_{21}^{(2)},\\
p_{31}&=X_{33}^{(1)},\\
p_{32}&=X_{33}^{(2)}-X_{23}^{(1)}X_{32}^{(1)}-X_{13}^{(1)}X_{31}^{(1)},\\
p_{33}&=X_{12}^{(1)}X_{23}^{(1)}X_{31}^{(1)}+X_{13}^{(1)}X_{21}^{(1)}X_{32}^{(1)}-X_{13}^{(1)}X_{31}^{(2)}-X_{23}^{(1)}X_{32}^{(2)}-X_{13}^{(2)}X_{31}^{(1)}-\\
	  &\ \ \ -X_{23}^{(2)}X_{32}^{(1)}-X_{13}^{(1)}X_{22}^{(1)}X_{31}^{(1)},\\
p_{34}&=-X_{23}^{(2)}X_{32}^{(2)}-X_{13}^{(2)}X_{31}^{(2)}-X_{13}^{(1)}X_{22}^{(1)}X_{31}^{(2)}+X_{12}^{(1)}X_{23}^{(1)}X_{31}^{(2)}+X_{13}^{(1)}X_{21}^{(1)}X_{32}^{(2)}-\\
	  &\ \ \ -X_{13}^{(1)}X_{22}^{(2)}X_{31}^{(1)}+X_{12}^{(1)}X_{23}^{(2)}X_{31}^{(1)}+X_{13}^{(1)}X_{21}^{(2)}X_{32}^{(1)}-X_{13}^{(2)}X_{22}^{(1)}X_{31}^{(1)}+\\
	  &\ \ \ +X_{12}^{(2)}X_{23}^{(1)}X_{31}^{(1)}+X_{13}^{(2)}X_{21}^{(1)}X_{32}^{(1)},\\
p_{35}&=-X_{13}^{(1)}X_{22}^{(2)}X_{31}^{(2)}+X_{12}^{(1)}X_{23}^{(2)}X_{31}^{(2)}+X_{13}^{(1)}X_{21}^{(2)}X_{32}^{(2)}-X_{13}^{(2)}X_{22}^{(1)}X_{31}^{(2)}+\\
	  &\ \ \ +X_{12}^{(2)}X_{23}^{(1)}X_{31}^{(2)}++X_{13}^{(2)}X_{21}^{(1)}X_{32}^{(2)}-X_{13}^{(2)}X_{22}^{(2)}X_{31}^{(1)}+X_{12}^{(2)}X_{23}^{(2)}X_{31}^{(1)}+\\
      &\ \ \ +X_{13}^{(2)}X_{21}^{(2)}X_{32}^{(1)},\\
p_{36}&=-X_{13}^{(2)}X_{22}^{(2)}X_{31}^{(2)}+X_{12}^{(2)}X_{23}^{(2)}X_{31}^{(2)}+X_{13}^{(2)}X_{21}^{(2)}X_{32}^{(2)}.\\
\end{align*} 
Analogously to previous proof we have that the decomposition in irreducibles components is

$$\gts=\bigcup_{i=1}^{22}C_i,$$
where

\begin{align*}
C_1   &=V\left(X_{11}^{(1)},X_{11}^{(2)},X_{22}^{(1)},X_{22}^{(2)}, X_{12}^{(1)},X_{12}^{(2)},X_{33}^{(1)},X_{33}^{(2)},X_{23}^{(1)},X_{23}^{(2)},X_{13}^{(1)},X_{13}^{(2)}\right),\\
C_2   &=V\bigg(X_{11}^{(1)},X_{11}^{(2)},X_{22}^{(1)},X_{22}^{(2)}, X_{12}^{(1)},X_{12}^{(2)},X_{33}^{(1)},X_{33}^{(2)}-X_{23}^{(1)}X_{32}^{(1)},X_{32}^{(2)},X_{23}^{(2)},\\
	  &\ \ \ \ \ \ \ \ X_{13}^{(1)},X_{13}^{(2)}\bigg),\\
C_3   &=V\left(X_{11}^{(1)},X_{11}^{(2)},X_{22}^{(1)},X_{22}^{(2)}, X_{12}^{(1)},X_{12}^{(2)},X_{33}^{(1)},X_{33}^{(2)},X_{32}^{(1)},X_{32}^{(2)},X_{13}^{(1)},X_{13}^{(2)}\right),\\
C_4   &=V\bigg(X_{11}^{(1)},X_{11}^{(2)},X_{22}^{(1)},X_{22}^{(2)}, X_{12}^{(1)},X_{12}^{(2)},X_{33}^{(1)},X_{33}^{(2)}-X_{23}^{(1)}X_{32}^{(1)}-X_{13}^{(1)}X_{31}^{(1)},\\
       &\ \ \ \ \ \ \ \ -X_{23}^{(2)}X_{32}^{(1)}+ X_{13}^{(1)}X_{21}^{(1)}X_{32}^{(1)}-X_{13}^{(1)}X_{31}^{(2)},X_{32}^{(2)},X_{21}^{(2)},X_{13}^{(2)}\bigg),\\
C_5   &=V\bigg(X_{11}^{(1)},X_{11}^{(2)},X_{22}^{(1)},X_{22}^{(2)}, X_{12}^{(1)},X_{12}^{(2)},X_{33}^{(1)},X_{33}^{(2)}-X_{23}^{(1)}X_{32}^{(1)}-X_{13}^{(1)}X_{31}^{(1)},\\
       &\ \ \ \ \ \ \ \ \ \ \ -X_{23}^{(1)}X_{32}^{(2)}-X_{13}^{(1)}X_{31}^{(2)} ,-X_{23}^{(2)}+ X_{13}^{(1)}X_{21}^{(1)},X_{21}^{(2)},X_{13}^{(2)}\bigg),\\
C_6   &=V\bigg(X_{11}^{(1)},X_{11}^{(2)},X_{22}^{(1)},X_{22}^{(2)}, X_{12}^{(1)},X_{12}^{(2)},X_{33}^{(1)},X_{33}^{(2)}-X_{13}^{(1)}X_{31}^{(1)},X_{31}^{(2)},X_{32}^{(1)},\\
       &\ \ \ \ \ \ \ \ X_{32}^{(2)},X_{13}^{(2)}\bigg),\\
C_7   &=V\bigg(X_{11}^{(1)},X_{11}^{(2)},X_{22}^{(1)},X_{22}^{(2)}, X_{12}^{(1)},X_{12}^{(2)},X_{33}^{(1)},X_{33}^{(2)}-X_{23}^{(1)}X_{32}^{(1)}-X_{13}^{(1)}X_{31}^{(1)},\\
       &\ \ \ \ \ \ \ \ \ X_{23}^{(1)}X_{32}^{(2)}+X_{23}^{(2)}X_{32}^{(1)}+X_{13}^{(1)}X_{31}^{(2)}+X_{13}^{(2)}X_{31}^{(1)},X_{23}^{(2)}X_{32}^{(2)}+X_{13}^{(2)}X_{31}^{(2)},\\
       &\ \ \ \ \ \ \ \ X_{21}^{(1)},X_{21}^{(2)}\bigg),\\ 
C_8   &=V\bigg(X_{11}^{(1)},X_{11}^{(2)},X_{22}^{(1)},X_{22}^{(2)}, X_{12}^{(1)},X_{12}^{(2)},X_{33}^{(1)},X_{33}^{(2)}-X_{23}^{(1)}X_{32}^{(1)}-X_{13}^{(1)}X_{31}^{(1)},\\
       &\ \ \ \ \ \ \ \ \ X_{23}^{(2)}X_{32}^{(1)}+X_{13}^{(2)}X_{31}^{(1)},X_{21}^{(1)}X_{32}^{(1)}-X_{31}^{(2)},X_{32}^{(2)},X_{21}^{(2)}\bigg),\\
C_9   &=V\bigg(X_{11}^{(1)},X_{11}^{(2)},X_{22}^{(1)},X_{22}^{(2)}, X_{12}^{(1)},X_{12}^{(2)},X_{33}^{(1)},X_{33}^{(2)},X_{31}^{(1)},X_{31}^{(2)},X_{32}^{(1)},X_{32}^{(2)}\bigg),\\
C_{10}&=V\bigg(X_{11}^{(1)},X_{11}^{(2)},X_{22}^{(1)}, X_{22}^{(2)}-X_{12}^{(1)}X_{21}^{(1)},X_{21}^{(2)},X_{12}^{(2)},X_{33}^{(1)},X_{33}^{(2)}-X_{13}^{(1)}X_{31}^{(1)},\\
       &\ \ \ \ \ \ \ \ X_{23}^{(1)},X_{23}^{(2)}- X_{13}^{(1)}X_{21}^{(1)},X_{31}^{(2)},X_{13}^{(2)}\bigg),\\
C_{11}&=V\bigg(X_{11}^{(1)},X_{11}^{(2)},X_{22}^{(1)}, X_{22}^{(2)}-X_{12}^{(1)}X_{21}^{(1)},X_{21}^{(2)},X_{12}^{(2)},X_{33}^{(1)},\\
       &\ \ \ \ \ \ \ \ X_{33}^{(2)}-X_{23}^{(1)}X_{32}^{(1)}-X_{13}^{(1)}X_{31}^{(1)},X_{32}^{(2)}-X_{12}^{(1)}X_{31}^{(1)}, X_{23}^{(2)}-X_{13}^{(1)}X_{21}^{(1)},X_{31}^{(2)},\\
       &\ \ \ \ \ \ \ \ X_{13}^{(2)}\bigg),\\
C_{12}&=V\bigg(X_{11}^{(1)},X_{11}^{(2)},X_{22}^{(1)}, X_{22}^{(2)}-X_{12}^{(1)}X_{21}^{(1)},X_{21}^{(2)},X_{12}^{(2)},X_{33}^{(1)},X_{33}^{(2)}-X_{13}^{(1)}X_{31}^{(1)},\\
       &\ \ \ \ \ \ \ \ X_{32}^{(1)},X_{32}^{(2)}-X_{12}^{(1)}X_{31}^{(1)},X_{31}^{(2)},X_{13}^{(2)}\bigg),\\
C_{13}&=V\bigg(X_{11}^{(1)},X_{11}^{(2)},X_{22}^{(1)}, X_{22}^{(2)}-X_{12}^{(1)}X_{21}^{(1)},X_{21}^{(2)},X_{12}^{(2)},X_{33}^{(1)},X_{33}^{(2)},X_{13}^{(1)},X_{23}^{(1)},\\
       &\ \ \ \ \ \ \ \ X_{23}^{(2)},X_{13}^{(2)}\bigg),\\
C_{14}&=V\bigg(X_{11}^{(1)},X_{11}^{(2)},X_{22}^{(1)},X_{21}^{(1)},X_{21}^{(2)},X_{12}^{(2)},X_{33}^{(1)},X_{33}^{(2)}-X_{23}^{(1)}X_{32}^{(1)}-X_{13}^{(1)}X_{31}^{(1)},\\
       &\ \ \ \ \ \ \ \ \ X_{23}^{(1)}X_{32}^{(2)}- X_{12}^{(1)}X_{23}^{(1)}X_{31}^{(1)}+X_{13}^{(2)}X_{31}^{(1)},X_{31}^{(2)},X_{23}^{(2)},X_{22}^{(2)}\bigg),\\
C_{15}&=V\bigg(X_{11}^{(1)},X_{11}^{(2)},X_{22}^{(1)},X_{21}^{(1)},X_{21}^{(2)},X_{12}^{(2)},X_{33}^{(1)},X_{33}^{(2)}-X_{23}^{(1)}X_{32}^{(1)}-X_{13}^{(1)}X_{31}^{(1)},\\
       &\ \ \ \ \ \ \ \ \ X_{23}^{(1)}X_{32}^{(2)}+ X_{13}^{(1)}X_{31}^{(2)},X_{12}^{(1)}X_{23}^{(1)}-X_{13}^{(2)},X_{23}^{(2)},X_{22}^{(2)}\bigg),\\
C_{16}&=V\bigg(X_{11}^{(1)},X_{11}^{(2)},X_{22}^{(1)},X_{21}^{(1)},X_{21}^{(2)},X_{12}^{(2)},X_{33}^{(1)},X_{33}^{(2)}-X_{23}^{(1)}X_{32}^{(1)}-X_{13}^{(1)}X_{31}^{(1)},\\
       &\ \ \ \ \ \ \ \ \ X_{23}^{(2)}X_{32}^{(1)}+X_{13}^{(2)}X_{31}^{(1)},X_{32}^{(2)}-X_{12}^{(1)}X_{31}^{(1)},X_{31}^{(2)},X_{22}^{(2)}\bigg),\\
C_{17}&=V\bigg(X_{11}^{(1)},X_{11}^{(2)},X_{22}^{(1)}, X_{22}^{(2)}-X_{12}^{(1)}X_{21}^{(1)},X_{21}^{(2)},X_{12}^{(2)},X_{33}^{(1)},X_{33}^{(2)},X_{31}^{(1)},X_{32}^{(1)},\\
       &\ \ \ \ \ \ \ \ X_{32}^{(2)},X_{31}^{(2)}\bigg),\\
C_{18}&=V\left(X_{11}^{(1)},X_{11}^{(2)},X_{22}^{(1)},X_{22}^{(2)}, X_{21}^{(1)},X_{21}^{(2)},X_{33}^{(1)},X_{33}^{(2)},X_{32}^{(1)},X_{32}^{(2)},X_{31}^{(1)},X_{31}^{(2)}\right),\\
C_{19}&=V\bigg(X_{11}^{(1)},X_{11}^{(2)},X_{22}^{(1)},X_{22}^{(2)}, X_{21}^{(1)},X_{21}^{(2)},X_{33}^{(1)},X_{33}^{(2)}-X_{32}^{(1)}X_{23}^{(1)},X_{23}^{(2)},X_{32}^{(2)},\\
       &\ \ \ \ \ \ \ \ X_{31}^{(1)},X_{31}^{(2)}\bigg),\\
C_{20}&=V\left(X_{11}^{(1)},X_{11}^{(2)},X_{22}^{(1)},X_{22}^{(2)}, X_{21}^{(1)},X_{21}^{(2)},X_{33}^{(1)},X_{33}^{(2)},X_{23}^{(1)},X_{23}^{(2)},X_{31}^{(1)},X_{31}^{(2)}\right),\\
C_{21}&=V\bigg(X_{11}^{(1)},X_{11}^{(2)},X_{22}^{(1)},X_{22}^{(2)}, X_{21}^{(1)},X_{21}^{(2)},X_{33}^{(1)},X_{33}^{(2)}-X_{31}^{(1)}X_{13}^{(1)},X_{13}^{(2)},X_{23}^{(1)},\\
       &\ \ \ \ \ \ \ \ X_{23}^{(2)},X_{31}^{(2)}\bigg),\\
C_{22}&=V\bigg(X_{11}^{(1)},X_{11}^{(2)},X_{22}^{(1)},X_{22}^{(2)}, X_{21}^{(1)},X_{21}^{(2)},X_{33}^{(1)},X_{33}^{(2)},X_{13}^{(1)},X_{13}^{(2)},X_{23}^{(1)},X_{23}^{(2)}\bigg).
\end{align*}
In this decomposition for $\gts$, each irreducible component has dimension $6$ and therefore the Gelfand-Tsetlin variety is equidimensional with 
$$\dim\gts=6.$$

\end{proof}


\subsection{A first decomposition} 

Now we will exhibit a decomposition of the Gelfand-Tsetlin variety $\gts$ for Yangians $Y_p(\gl_n)$, which as we will see, such decomposition is not in irreducible components.


\begin{proposition}
\label{decomp1}
For $Y_p(\gl_n)$ 
$$\gts=\bigcup_{s=1}^{p+1}\gts_s,$$
where
$$\gts_1:=V\left(\left\{p_{ij}\right\}_{\substack{i=1,2\\ j=1,2,\ldots,p
}}\cup\left\{X_{12}^{(i)}\right\}_{i=1}^{p}\cup\left\{p_{ij}\right\}_{\substack{i=3,4,\ldots,n\\ j=1,2,\ldots,ip
}}\right)\subset \k^{n^2p},$$
for $2\leq s\leq p,$
$$\gts_s:=V\left(\left\{p_{ij}\right\}_{\substack{i=1,2\\ j=1,2,\ldots,p
}}\cup\left\{X_{12}^{(i)}\right\}_{i=s}^{p}\cup\left\{X_{21}^{(i)}\right\}_{i=p-(s-2)}^{p}\cup\left\{p_{ij}\right\}_{\substack{i=3,4,\ldots,n\\ j=1,2,\ldots,ip
}}\right)\subset \k^{n^2p}$$
and
$$\gts_{p+1}:=V\left(\left\{p_{ij}\right\}_{\substack{i=1,2\\ j=1,2,\ldots,p
}}\cup\left\{X_{21}^{(i)}\right\}_{i=1}^{p}\cup\left\{p_{ij}\right\}_{\substack{i=3,4,\ldots,n\\ j=1,2,\ldots,ip
}}\right)\subset \k^{n^2p}.$$
\end{proposition}


\begin{proof}
Clearly $\gts_s\subseteq\gts$ for each $s=1,2,\ldots,p+1$, therefore
$$\gts\supseteq\bigcup_{s=1}^{p+1}\gts_s.$$
Now, for 
$$\gts\subseteq\bigcup_{s=1}^{p+1}\gts_s,$$
we consider 
$A=\left(a_{ij}^{(t)}\right)_{\substack{i,j=1,2,\ldots,n\\ t=1,2,\ldots,p}}\in\gts$. By Proposition \eqref{equivGTs} or Corollary \eqref{equivGTsYp(n)} $p_{2\,2p}(A)=a_{12}^{(p)}a_{21}^{(p)}=0,$
hence 
$$a_{12}^{(p)}=0\ \ \mbox{or}\ \ a_{21}^{(p)}=0.$$
Therefore, if $p=1$ then
$$\gts=\gts_1\cup\gts_2.$$

When $p>1$ as a consequence of the Proposition \eqref{equivGTs} or Corollary \eqref{equivGTsYp(n)}
 $$p_{2\,2p-1}(A)=a_{12}^{(p-1)}a_{21}^{(p)}+a_{12}^{(p)}a_{21}^{(p-1)}=0,$$
thus
 $$a_{12}^{(p)}=a_{12}^{(p-1)}=0\ \ \mbox{or}\ \ a_{12}^{(p)}=a_{21}^{(p)}=0\ \ \mbox{or}\ \ a_{21}^{(p)}=a_{21}^{(p-1)}=0.$$
Therefore, if $p=2$ then
 $$\gts=\gts_1\cup\gts_2\cup\gts_3.$$

Moreover, for the case that $p>2$. Since that $A\in\gts_1$ whenever $a_{12}^{(t)}=0$ for all $t=1,2,\ldots,p$ and $A\in\gts_{p+1}$ whenever $a_{21}^{(t)}=0$ for all $t=1,2,\ldots,p$, then we can assume that, there exists $s\in\left\{2,3,\ldots,p\right\}$ such that 
 $$a_{12}^{(t)}=0,\ \ \forall t\geq s\ \ \ \mbox{and}\ \ \ a_{12}^{(s-1)}\neq0.$$
By Proposition \eqref{equivGTs} or Corollary \eqref{equivGTsYp(n)} follows
 $$p_{p\,p+s-1}(A)=a_{12}^{(s-1)}a_{21}^{(p)}+\sum_{t=s}^{p}a_{12}^{(t)}a_{21}^{(p+s-1-t)}=0$$
and thus $a_{21}^{(p)}=0$, which implies that, if $s=2$ then $A\in\gts_2$.
  
Continuing with the same argument for $3\leq s\leq p$, when  
 $$a_{21}^{(t)}=0,\ \ \forall t\geq j\ \ \ \mbox{for some}\ \ \ j\in\left\{p-s+3,p-s+4,\ldots,p\right\}.$$
Since that $1\leq s-p+j-2\leq p$, we have the polynomial
 \begin{align*}
 p_{p\,s+j-2}&=p_{p\,p+s-p+j-2}=\sum_{t=s-p+j-2}^{p}X_{12}^{(t)}X_{21}^{(p+s-p+j-2-t)}\\
            &= \sum_{m=j}^{p}X_{12}^{(s+j-2-m)}X_{21}^{(m)}+X_{12}^{(s-1)}X_{21}^{(j-1)}+\sum_{t=s}^{p}X_{12}^{(t)}X_{21}^{(s+j-2-t)},
 \end{align*}
hence, of the Proposition \eqref{equivGTs} or Corollary \eqref{equivGTsYp(n)},
 \begin{align*}
 p_{p\,s+j-2}(A)&= \underbrace{\sum_{m=j}^{p}a_{12}^{(s+j-2-m)}a_{21}^{(m)}}_{=0}+a_{12}^{(s-1)}a_{21}^{(j-1)}+\underbrace{\sum_{t=s}^{p}a_{12}^{(t)}a_{21}^{(s+j-2-t)}}_{=0}=0,
 \end{align*}
thus $a_{21}^{(j-1)}=0$ and we can conclude $A\in\gts_s,\ \ 3\leq s\leq p.$

\end{proof}


\subsubsection{A decomposition for the first subvariety}

Note that as 
 $$\gts_1=V\left(\left\{p_{ij}\right\}_{\substack{i=1,2\\ j=1,2,\ldots,p}}\cup\left\{X_{12}^{(i)}\right\}_{i=1}^{p}\cup\left\{p_{ij}\right\}_{\substack{i=3,4,\ldots,n\\ j=1,2,\ldots,ip}}\right)\subset \k^{n^2p},$$
so by Corollary \eqref{equivGTsYp(n)}, we can see that variety of the following form
 $$\gts_1=V\left(\textbf{W}\cup\left\{p_{ij}\right\}_{\substack{i=3,4,\ldots,n\\ j=1,2,\ldots,ip}}\right)\subset \k^{n^2p},$$
where
 $$\textbf{W}:=\bigcup_{i=1}^{p}\left\{X_{11}^{(i)},X_{12}^{(i)},X_{22}^{(i)}\right\}.$$


\vspace{0.3cm}

\noindent When $n=3$, we will call $\gts_1$ \textbf{weak version of the Gelfand-Tsetlin variety} $\gts$ for $Y_p(\gl_3)$. In level $p=1$, this subvariety coincides with the weak version in \cite{gbm}.


\vspace{0.7cm}

Now, we going to looking for a decomposition (it is not necessarily in irreducible components) to $\gts_1$. With this aim, we will need of the following lemma:


\vspace{0.4cm}

\begin{lemma}
\label{lemmarecursN1}
Suppose $A=\left(a_{rs}^{(t)}\right)_{\substack{r,s=1,2,\ldots,n\\ t=1,2,\ldots,p
}}\in\gts_1$ and $1\leq i\leq p-2$

\begin{enumerate}

\item
\begin{enumerate}
	  
\item If $a_{13}^{(t)}=0,\ \forall t\geq p-i$, then

$$a_{13}^{(p-i-1)}=0\ \ \mbox{or}\ \ a_{21}^{(p)}=0\ \ \mbox{or}\ \ a_{32}^{(p)}=0.$$
	  
\item If $1\leq j\leq i$, $a_{13}^{(t)}=0,\ \forall t\geq p-i+j$ and $a_{21}^{(t)}=0,\ \forall t\geq p-j+1$, then 
 
 $$a_{13}^{(p-i+j-1)}=0\ \ \mbox{or}\ \ a_{21}^{(p-j)}=0\ \ \mbox{or}\ \ a_{32}^{(p)}=0.$$
    
\item If $a_{21}^{(t)}=0,\ \forall t\geq p-i$, then

$$a_{13}^{(p)}=0\ \ \mbox{or}\ \ a_{21}^{(p-i-1)}=0\ \ \mbox{or}\ \ a_{32}^{(p)}=0.$$
    
\end{enumerate}

\item
\begin{enumerate}
	
\item If $1\leq j\leq i$, $a_{21}^{(t)}=0,\ \forall t\geq p-i+j$ and $a_{32}^{(t)}=0,\ \forall t\geq p-j+1$, then
$$a_{13}^{(p)}=0\ \ \mbox{or}\ \ a_{21}^{(p-i+j-1)}=0\ \ \mbox{or}\ \ a_{32}^{(p-j)}=0.$$
	
\item If $a_{32}^{(t)}=0,\ \forall t\geq p-i$, then
$$a_{13}^{(p)}=0\ \ \mbox{or}\ \ a_{21}^{(p)}=0\ \ \mbox{or}\ \ a_{32}^{(p-i-1)}=0.$$
\end{enumerate}

\item If $1\leq j\leq i$, $a_{13}^{(t)}=0,\ \forall t\geq p-i+j$ and $a_{32}^{(t)}=0,\ \forall t\geq p-j+1$, then
$$a_{13}^{(p-i+j-1)}=0\ \ \mbox{or}\ \ a_{21}^{(p)}=0\ \ \mbox{or}\ \ a_{32}^{(p-j)}=0.$$

\end{enumerate}

\end{lemma}


\begin{proof}
We will prove only the statement for the case $(1)$, because the cases $(2)$ and $(3)$ can be treated analogously.

Since that $1\leq i\leq p-2$, then $1\leq p-i-1\leq p-2$. We can consider 
\begin{align*}
p_{3\,3p-i-1}&=\sum_{r=p-i-1}^{p}\sum_{s=2p-i-1-r}^{p} \biggl\{-X_{13}^{(r)}X_{22}^{(s)}X_{31}^{(3p-i-1-r-s)}+\\
&\ \ \ +X_{12}^{(r)}X_{23}^{(s)}X_{31}^{(3p-i-1-r-s)}+X_{13}^{(r)}X_{21}^{(s)}X_{32}^{(3p-i-1-r-s)}\biggr\},
\end{align*}
as $A\in\gts_1$, by Corollary \eqref{equivGTsYp(n)} we have the equation \ $p_{3\,3p-i-1}(A)=0$

\begin{align*}
0&=\sum_{r=p-i-1}^{p}\sum_{s=2p-i-1-r}^{p} \biggl\{\underbrace{-a_{13}^{(r)}a_{22}^{(s)}a_{31}^{(3p-i-1-r-s)}+a_{12}^{(r)}a_{23}^{(s)}a_{31}^{(3p-i-1-r-s)}}_{=0}+\\
&\ \ \ +a_{13}^{(r)}a_{21}^{(s)}a_{32}^{(3p-i-1-r-s)}\biggr\}\\
	       0&=\sum_{r=p-i-1}^{p}\sum_{s=2p-i-1-r}^{p} a_{13}^{(r)}a_{21}^{(s)}a_{32}^{(3p-i-1-r-s)}\\
\end{align*}

\begin{enumerate}[$(a)$]
	  
	  \item If $a_{13}^{(t)}=0,\ \forall t\geq p-i$. Then
\begin{align*}
p_{3\, 3p-i-1}(A)&= a_{13}^{(p-i-1)}a_{21}^{(p)}a_{32}^{(p)}+\underbrace{\sum_{r=p-i}^{p}\sum_{s=2p-i-1-r}^{p} a_{13}^{(r)}a_{21}^{(s)}a_{32}^{(3p-i-1-r-s)}}_{=0}\\
                0&=a_{13}^{(p-i-1)}a_{21}^{(p)}a_{32}^{(p)},
\end{align*}
hence

 $$a_{13}^{(p-i-1)}=0\ \ \ \mbox{or}\ \ \ a_{21}^{(p)}=0\ \ \ \mbox{or}\ \ \ a_{32}^{(p)}=0.$$
	  
	  \item If $1\leq j\leq i$, $a_{13}^{(t)}=0,\ \forall t\geq p-i+j$ and $a_{21}^{(t)}=0, \forall t\geq p-j+1$. Then 
 \begin{align*}
 0 &= \underbrace{\sum_{r=p-i-1}^{p-i+j-2}\sum_{s=2p-i-1-r}^{p} a_{13}^{(r)}a_{21}^{(s)}a_{32}^{(3p-i-1-r-s)}}_{=0}+\sum_{s=p-j}^{p} a_{13}^{(p-i+j-1)}a_{21}^{(s)}a_{32}^{(2p-j-s)}+\\
                 &\ \ \ +\underbrace{\sum_{r=p-i+j}^{p}\sum_{s=2p-i-1-r}^{p} a_{13}^{(r)}a_{21}^{(s)}a_{32}^{(3p-i-1-r-s)}}_{=0}\\
                0&=\sum_{s=p-j}^{p} a_{13}^{(p-i+j-1)}a_{21}^{(s)}a_{32}^{(2p-j-s)}\\                                         
                0&= a_{13}^{(p-i+j-1)}a_{21}^{(p-j)}a_{32}^{(p)}+\underbrace{\sum_{s=p-j+1}^{p} a_{13}^{(p-i+j-1)}a_{21}^{(s)}a_{32}^{(2p-j-s)}}_{=0}\\ 
                0&=a_{13}^{(p-i+j-1)}a_{21}^{(p-j)}a_{32}^{(p)}.
 \end{align*}
Thus
\[a_{13}^{(p-i+j-1)}=0\ \ \mbox{or}\ \ a_{21}^{(p-j)}=0\ \ \mbox{or}\ \ a_{32}^{(p)}=0.\]
    
    \item If $a_{21}^{(t)}=0,\ \ \forall t\geq p-i$. Then
\begin{align*}
                0&= \underbrace{\sum_{r=p-i-1}^{p-1}\sum_{s=2p-i-1-r}^{p} a_{13}^{(r)}a_{21}^{(s)}a_{32}^{(3p-i-1-r-s)}}_{=0}+\sum_{s=p-i-1}^{p} a_{13}^{(p)}a_{21}^{(s)}a_{32}^{(2p-i-1-s)}\\
                0&=\sum_{s=p-i-1}^{p} a_{13}^{(p)}a_{21}^{(s)}a_{32}^{(2p-i-1-s)}\\
                0&=a_{13}^{(p)}a_{21}^{(p-i-1)}a_{32}^{(p)}+ \underbrace{\sum_{s=p-i}^{p}a_{13}^{(p)}a_{21}^{(s)}a_{32}^{(2p-i-1-s)}}_{=0}\\
                0&=a_{13}^{(p)}a_{21}^{(p-i-1)}a_{32}^{(p)}.
\end{align*}
Therefore
$$a_{13}^{(p)}=0\ \ \ \mbox{or}\ \ \ a_{21}^{(p-i-1)}=0\ \ \ \mbox{or}\ \ \ a_{32}^{(p)}=0.$$
    
\end{enumerate}

\end{proof}


\begin{notation}
\label{subvarWeak}
For $Y_p({\gl_n})$ we will consider the following subvarieties of $\gts_1$:
$$\wk^p:=V\left(\textbf{W}\cup\left\{p_{3j}\right\}_{j=1}^{3p-3}\cup\left\{X_{32}^{(p)},X_{21}^{(p)},X_{13}^{(p)}\right\}\cup\left\{p_{ij}\right\}_{\substack{i=4,5,\ldots,n\\ j=1,2,\ldots,ip}}\right).$$

$$\wk_1:=\bigcup_{s=1}^{p+1}\wk_{1s}.$$
where
\begin{itemize}

  \item
$$\wk_{11}:=V\left(\textbf{W}\cup\left\{p_{3j}\right\}_{j=1}^{2p}\cup\left\{X_{13}^{(t)}\right\}_{t=1}^{p}\cup\left\{p_{ij}\right\}_{\substack{i=4,5,\ldots,n\\ j=1,2,\ldots,ip}}\right).$$

  \item For each $s=2,3,\ldots,p$,
$$\wk_{1s}:=V\left(\textbf{W}\cup\left\{p_{3j}\right\}_{j=1}^{2p}\cup\left\{X_{13}^{(t)}\right\}_{t=s}^{p}\cup\left\{X_{21}^{(t)}\right\}_{t=p-(s-2)}^{p}\cup\left\{p_{ij}\right\}_{\substack{i=4,5,\ldots,n\\ j=1,2,\ldots,ip}}\right).$$

  \item 
$$\wk_{1\,p+1}:=V\left(\textbf{W}\cup\left\{p_{3j}\right\}_{j=1}^{2p}\cup\left\{X_{21}^{(t)}\right\}_{t=1}^{p}\cup\left\{p_{ij}\right\}_{\substack{i=4,5,\ldots,n\\ j=1,2,\ldots,ip}}\right).$$
\end{itemize}
Also,
$$\wk_2:=\bigcup_{s=2}^{p+1}\wk_{2s}.$$
where
\begin{itemize}

  \item For each $s=2,3,\ldots,p$,
$$\wk_{2s}:=V\left(\textbf{W}\cup\left\{p_{3j}\right\}_{j=1}^{2p}\cup\left\{X_{21}^{(t)}\right\}_{t=s}^{p}\cup\left\{X_{32}^{(t)}\right\}_{t=p-(s-2)}^{p}\cup\left\{p_{ij}\right\}_{\substack{i=4,5,\ldots,n\\ j=1,2,\ldots,ip}}\right).$$
 
  \item 
$$\wk_{2\,p+1}:=V\left(\textbf{W}\cup\left\{p_{3j}\right\}_{j=1}^{2p}\cup\left\{X_{32}^{(t)}\right\}_{t=1}^{p}\cup\left\{p_{ij}\right\}_{\substack{i=4,5,\ldots,n\\ j=1,2,\ldots,ip}}\right).$$
\end{itemize}     
And,
$$\wk_3:=\bigcup_{s=2}^{p}\wk_{3s}.$$
where
\begin{itemize}
  \item For each $s=2,3,\ldots,p$,
$$\wk_{3s}:=V\left(\textbf{W}\cup\left\{p_{3j}\right\}_{j=1}^{2p}\cup\left\{X_{13}^{(t)}\right\}_{t=s}^{p}\cup\left\{X_{32}^{(t)}\right\}_{t=p-(s-2)}^{p}\cup\left\{p_{ij}\right\}_{\substack{i=4,5,\ldots,n\\ j=1,2,\ldots,ip}}\right).$$
  
\end{itemize}
\end{notation}

\vspace{0.2cm}

\begin{proposition}
\label{firstdecompGTs1}
For $Y_p(\gl_n)$
 $$\gts_1=\wk_{1}\cup\wk_{2},\ \ \mbox{whenever}\ \ p=1;$$
 $$\gts_1=\wk_{1}\cup\wk_{2}\cup\wk_{3},\ \ \mbox{whenever}\ \ p=2;$$
 $$\gts_1=\wk^{p}\cup\wk_{1}\cup\wk_{2}\cup\wk_{3},\ \ \mbox{whenever}\ \ p\geq 3.$$
\end{proposition}


\begin{proof}
Clearly, each $\wk$'s is contained in $\gts_1$, then we only need to prove
 $$\gts_1\subseteq\wk_{1}\cup\wk_{2},\ \ \mbox{for}\ \ p=1,$$
 $$\gts_1\subseteq\wk_{1}\cup\wk_{2}\cup\wk_{3},\ \ \mbox{for}\ \ p=2,$$
 $$\gts_1\subseteq\wk^{p}\cup\wk_{1}\cup\wk_{2}\cup\wk_{3},\ \ \mbox{for}\ \ p\geq 3.$$
  In fact, we consider 
$A=\left(a_{ij}^{(t)}\right)_{\substack{i,j=1,2,\ldots,n\\ t=1,2,\ldots,p}}\in\gts_1$. 
For this proof we will use

\begin{align*}
p_{3\, 2p+i}&=\sum_{r=i}^{p}\sum_{s=p+i-r}^{p} \biggl\{-X_{13}^{(r)}X_{22}^{(s)}X_{31}^{(2p+i-r-s)}+X_{12}^{(r)}X_{23}^{(s)}X_{31}^{(2p+i-r-s)}+\\
			&\ \ \ +X_{13}^{(r)}X_{21}^{(s)}X_{32}^{(2p+i-r-s)}\biggr\},\ \ \ i=p-1,p.
\end{align*}
By Corollary \eqref{equivGTsYp(n)},
\begin{align*}
p_{3\,3p}(A)&=-\underbrace{a_{13}^{(p)}a_{22}^{(p)}a_{31}^{(p)}+a_{12}^{(p)}a_{23}^{(p)}a_{31}^{(p)}}_{=0}+a_{13}^{(p)}a_{21}^{(p)}a_{32}^{(p)}=0,
\end{align*}
hence $a_{13}^{(p)}=0\ \ \ \mbox{or}\ \ \ a_{21}^{(p)}=0\ \ \ \mbox{or}\ \ \ a_{32}^{(p)}=0,$
and as a consequence, when $p=1$ we have $A\in\wk_{1}\cup\wk_{2}$. 

For $p>1$, follows of the Corollary \eqref{equivGTsYp(n)}

\begin{align*}
p_{3\,3p-1}(A)&=\sum_{r=p-1}^{p}\sum_{s=2p-1-r}^{p} \biggl\{\underbrace{-a_{13}^{(r)}a_{22}^{(s)}a_{31}^{(3p-1-r-s)}+a_{12}^{(r)}a_{23}^{(s)}a_{31}^{(3p-1-r-s)}}_{=0}+\\
			  &\ \ \ +a_{13}^{(r)}a_{21}^{(s)}a_{32}^{(3p-1-r-s)}\biggr\}\\
             0&=\sum_{r=p-1}^{p}\sum_{s=2p-1-r}^{p} a_{13}^{(r)}a_{21}^{(s)}a_{32}^{(3p-1-r-s)}\\
             0&=a_{13}^{(p-1)}a_{21}^{(p)}a_{32}^{(p)} +a_{13}^{(p)}a_{21}^{(p-1)}a_{32}^{(p)}+a_{13}^{(p)}a_{21}^{(p)}a_{32}^{(p-1)}=0,\\
\end{align*}
hence
$$a_{13}^{(p)}=a_{13}^{(p-1)}=0\ \ \mbox{or}\ \ a_{13}^{(p)}=a_{21}^{(p)}=0\ \ \mbox{or}\ \ a_{13}^{(p)}=a_{32}^{(p)}=0$$
or
$$a_{21}^{(p)}=a_{21}^{(p-1)}=0\ \ \mbox{or}\ \ a_{21}^{(p)}=a_{32}^{(p)}=0$$
or
$$a_{32}^{(p)}=a_{32}^{(p-1)}=0,$$

\noindent thus, in the case $p=2$ we have $A\in\wk_{1}\cup\wk_{2}\cup\wk_{3}$. 

Now, when $p>2$ by lemma \eqref{lemmarecursN1} we have the cases:

$$a_{13}^{(p)}=a_{13}^{(p-1)}=a_{13}^{(p-2)}=0\ \ \mbox{or}\ \ a_{13}^{(p)}=a_{13}^{(p-1)}=a_{21}^{(p)}=0\ \ \mbox{or}\ \ a_{13}^{(p)}=a_{13}^{(p-1)}=a_{32}^{(p)}=0$$
or
$$a_{13}^{(p)}=a_{21}^{(p)}=a_{21}^{(p-1)}=0\ \ \mbox{or}\ \ a_{13}^{(p)}=a_{21}^{(p)}=a_{32}^{(p)}=0$$
or
$$a_{13}^{(p)}=a_{32}^{(p)}=a_{32}^{(p-1)}=0$$
or
$$a_{21}^{(p)}=a_{21}^{(p-1)}=a_{21}^{(p-2)}=0\ \ \mbox{or}\ \ a_{21}^{(p)}=a_{21}^{(p-1)}=a_{32}^{(p)}=0$$
or
$$a_{21}^{(p)}=a_{32}^{(p)}=a_{32}^{(p-1)}=0$$
or
$$a_{32}^{(p)}=a_{32}^{(p-1)}=a_{32}^{(p-2)}=0.$$

Whenever 
$$a_{13}^{(p)}=a_{21}^{(p)}=a_{32}^{(p)}=0,$$
we have $A\in\wk^p$ and therefore we can assume that 

$$A\not\in\wk^p.$$
Similarly, when 
$$a_{13}^{(t)}=0,\ \ \mbox{for all}\ \ t=1,2,\ldots,p$$
or
$$a_{21}^{(t)}=0,\ \ \mbox{for all}\ \ t=1,2,\ldots,p$$
or
$$a_{32}^{(t)}=0,\ \ \mbox{for all}\ \ t=1,2,\ldots,p$$
we have that $A\in\wk_{11}\cup\wk_{1\,p+1}\cup\wk_{2\,p+1}$ and thus we can assume that 

$$A\not\in\wk_{11}\cup\wk_{1\,p+1}\cup\wk_{2\,p+1}.$$
Continuing with the same argument, suppose that

$$A\not\in\wk^p\cup\wk_{1}\cup\wk_{2}$$
and we will prove that $A\in\wk_{3}.$

Clearly, there exists $s\in\left\{2,3,\ldots,p\right\}$ such that
$$a_{13}^{(t)}=0,\ \ \mbox{for each}\ \ t=s,s+1,\ldots,p\ \ \mbox{and}\ \ a_{13}^{(s-1)}\neq0$$
because $A\not\in\wk_{11}$, follows from the Lemma \eqref{lemmarecursN1} item $(1.a)$ 
$$a_{21}^{(p)}=0\ \ \mbox{or}\ \ a_{32}^{(p)}=0.$$
We note that, if $a_{21}^{(p)}=0$ there exists 
$j\in\left\{p-(s-3),p-(s-4),\ldots,p\right\}$ such that 
$$a_{21}^{(t)}=0,\ \ \mbox{for all}\ \ t=j,j+1,\ldots,p\ \ \mbox{and}\ \ a_{21}^{(j-1)}\neq0$$
because $A\not\in\wk_{1s}$ and as a consequence of the Lemma \eqref{lemmarecursN1} item $(1.b)$, we have
$$a_{32}^{(p)}=0$$
which is a contradiction with the fact that $A\not\in\wk^p$. Therefore  
$$a_{21}^{(p)}\neq0\ \ \mbox{and}\ \ a_{32}^{(p)}=0.$$
Again, by Lemma \eqref{lemmarecursN1} item $(3)$ we have $a_{32}^{(p-1)}=0$ and continuing of this way we have that 
$$A\in\wk_{3s}\subseteq\wk_3.$$

\end{proof}


\vspace{0.3cm}

Now, we going to improve the previous descomposition from the Proposition \eqref{firstdecompGTs1} and with that aim, we will need the following Lemma:


\begin{lemma}
\label{lemmarecursN2}
Suppose $A=\left(a_{ij}^{(t)}\right)_{\substack{i,j=1,2,\ldots,n\\ t=1,2,\ldots,p}}\in\gts_1$, $2\leq s\leq p-1$ and $0\leq j\leq p-s$. 

If $a_{13}^{(t)}=0,\ \forall t\geq s$ and 
$$a_{32}^{(p)}=a_{32}^{(p-1)}=a_{32}^{(p-2)}=\cdots=a_{32}^{(p-(s-2))}=a_{32}^{(p-(s-2)-1)}=\cdots=a_{32}^{(p-(s-2)-j)}=0,$$
then
$$a_{13}^{(s-1)}=0\ \ \mbox{or}\ \ a_{21}^{(p)}=0\ \ \mbox{or}\ \ a_{32}^{(p-(s-2)-j-1)}=0.$$

\end{lemma}


\begin{proof}
Since that $A\in\gts_1$, $s\leq p-j$ and $p-(s-2)-j\leq p-j$ we have the equation 
\begin{align*}
p_{3\,2p-j}(A)&=-\underbrace{\sum_{t=p-j}^pa_{23}^{(t)}a_{32}^{(2p-j-t)}}_{=0}-\underbrace{\sum_{t=p-j}^pa_{13}^{(t)}a_{31}^{(2p-j-t)}}_{=0}+\\
              &\ \ \ +\sum_{r=1}^{p-j-1}\sum_{t=p-j-r}^p\biggl\{\underbrace{-a_{13}^{(r)}a_{22}^{(t)}a_{31}^{(2p-j-r-t)}+a_{12}^{(r)}a_{23}^{(t)}a_{31}^{(2p-j-r-t)}}_{=0}+\\
              &\ \ \ +a_{13}^{(r)}a_{21}^{(t)}a_{32}^{(2p-j-r-t)}\biggr\}+\\
              &\ \ \ +\sum_{r=p-j}^{p}\sum_{t=1}^{2p-j-r-1}\biggl\{\underbrace{-a_{13}^{(r)}a_{22}^{(t)}a_{31}^{(2p-j-r-t)}+a_{12}^{(r)}a_{23}^{(t)}a_{31}^{(2p-j-r-t)}}_{=0}+\\
              &\ \ \ +\underbrace{a_{13}^{(r)}a_{21}^{(t)}a_{32}^{(2p-j-r-t)}}_{=0}\biggr\}\\
             0&= \sum_{r=1}^{p-j-1}\sum_{t=p-j-r}^pa_{13}^{(r)}a_{21}^{(t)}a_{32}^{(2p-j-r-t)}.
\end{align*}

If $s=2$ 
\begin{align*}
p_{3\,2p-j}(A)&= \sum_{t=p-j-1}^{p}a_{13}^{(1)}a_{21}^{(t)}a_{32}^{(2p-j-1-t)}+\underbrace{\sum_{r=2}^{p-j-1}\sum_{t=p-j-r}^{p}a_{13}^{(r)}a_{21}^{(t)}a_{32}^{(2p-j-r-t)}}_{=0}\\
             0&= \sum_{t=p-j-1}^{p}a_{13}^{(1)}a_{21}^{(t)}a_{32}^{(2p-j-1-t)}.
\end{align*}

If $2<s<p-j$
\begin{align*}          
p_{3\,2p-j}(A)&= \sum_{r=1}^{s-2}\sum_{t=p-j-r}^{p}a_{13}^{(r)}a_{21}^{(t)}a_{32}^{(2p-j-r-t)}+\sum_{t=p-j-s+1}^{p}a_{13}^{(s-1)}a_{21}^{(t)}a_{32}^{(2p-j-s+1-t)}+\\
                         &\ \ \ +\underbrace{\sum_{r=s}^{p-j-1}\sum_{t=p-j-r}^{p}a_{13}^{(r)}a_{21}^{(t)}a_{32}^{(2p-j-r-t)}}_{=0}\\     
                       0 &= \underbrace{\sum_{r=1}^{s-2}\sum_{m=p-j-r}^{p}a_{13}^{(r)}a_{21}^{(2p-j-r-m)}a_{32}^{(m)}}_{=0}+\sum_{t=p-j-s+1}^{p}a_{13}^{(s-1)}a_{21}^{(t)}a_{32}^{(2p-j-s+1-t)}\\
                       0 &= \sum_{t=p-j-s+1}^{p}a_{13}^{(s-1)}a_{21}^{(t)}a_{32}^{(2p-j-s+1-t)}.
\end{align*}

If $s=p-j$
\begin{align*}
p_{3\,2p-j}(A)&= \sum_{r=1}^{p-j-2}\sum_{t=p-j-r}^{p}a_{13}^{(r)}a_{21}^{(t)}a_{32}^{(2p-j-r-t)}+\sum_{t=1}^{p}a_{13}^{(p-j-1)}a_{21}^{(t)}a_{32}^{(p+1-t)}\\
             0&= \underbrace{\sum_{r=1}^{p-j-2}\sum_{m=p-j-r}^{p}a_{13}^{(r)}a_{21}^{(2p-j-r-m)}a_{32}^{(m)}}_{=0}+\sum_{t=1}^{p}a_{13}^{(p-j-1)}a_{21}^{(t)}a_{32}^{(p+1-t)}\\
             0&= \sum_{t=1}^{p}a_{13}^{(p-j-1)}a_{21}^{(t)}a_{32}^{(p+1-t)}.
\end{align*}

Now, by all previous cases we have
\begin{align*}                          
p_{3\,2p-j}(A)&= \sum_{t=p-j-s+1}^{p}a_{13}^{(s-1)}a_{21}^{(t)}a_{32}^{(2p-j-s+1-t)}=0\\
              &= \sum_{t=p-j-s+1}^{p-1}a_{13}^{(s-1)}a_{21}^{(t)}a_{32}^{(2p-j-s+1-t)}+a_{13}^{(s-1)}a_{21}^{(p)}a_{32}^{(p-j-s+1)}=0\\
              &= \underbrace{\sum_{m=p-j-s+2}^{p}a_{13}^{(s-1)}a_{21}^{(2p-j-s+1-m)}a_{32}^{(m)}}_{=0}+a_{13}^{(s-1)}a_{21}^{(p)}a_{32}^{(p-j-s+1)}=0\\
              &=a_{13}^{(s-1)}a_{21}^{(p)}a_{32}^{(p-j-s+1)}=0.
\end{align*}

\end{proof}


\begin{proposition}
\label{W3cWpuW1uW2}
For $Y_p(\gl_n)$ with $p>1$
$$\wk_3\subseteq \wk^p\cup\wk_1\cup\wk_2.$$
\end{proposition}


\begin{proof}
Suppose $s\in\left\{2,3,\ldots,p\right\}$ and 
$A=\left(a_{ij}^{(t)}\right)_{\substack{i,j=1,2,\ldots,n\\ t=1,2,\ldots,p}}\in\wk_{3s}$, 
this implies
$$a_{13}^{(p)}=a_{13}^{(p-1)}=\cdots=a_{13}^{(s)}=0,$$
$$a_{32}^{(p)}=a_{32}^{(p-1)}=\cdots=a_{32}^{(p-(s-2))}=0.$$
By Lemma \eqref{lemmarecursN2}
$$a_{13}^{(s-1)}=0\ \ \mbox{or}\ \ a_{21}^{(p)}=0\ \ \mbox{or}\ \ a_{32}^{(p-(s-2)-1)}=0.$$

If $a_{21}^{(p)}=0$, then $A\in\wk^p$.

If $a_{13}^{(t)}=0\ \ \forall t=1,2,\ldots,p$, then $A\in\wk_{11}\subseteq \wk_1$.

If there exists $t\in\left\{2,3,\ldots,s\right\}$ such that 
$$a_{13}^{(j)}=0\ \ \forall j=t,t+1,\ldots,p \ \ \mbox{and}\ \ \ a_{13}^{(t-1)}\neq0,$$
then of the Lemma \eqref{lemmarecursN2} 
$$A\in\wk_{2\,p+1}\subseteq \wk_2.$$

\end{proof}


\begin{corollary}
\label{decompweak}
For $Y_p(\gl_n)$
 $$\gts_1=\wk_{1}\cup\wk_{2},\ \ \mbox{whenever}\ \ p=1,2;$$
 $$\gts_1=\wk^{p}\cup\wk_{1}\cup\wk_{2},\ \ \mbox{whenever}\ \ p\geq 3.$$
\end{corollary}


\begin{proof}
 It follows of the Propositions \eqref{firstdecompGTs1} and \eqref{W3cWpuW1uW2}.

 \end{proof}


\subsection{Equidimensionality for the weak version of the Gelfand-Tsetlin}

We will prove that the weak version $\gts_1$ of the $\gts$ to $Y_p(\gl_3)$ is equidimensional of dimension $\dim\gts_1=3p$, showing the equidimensionality of every subvariety $\wk$'s in the decomposition of the Corollary  \eqref{decompweak}. With this aim the following Theorem of V. Futorny, A. Molev and S. Ovsienko will be very useful.


\begin{theorem}
\label{EquidGTsYp(2)}
The Gelfand-Tsetlin variety $\gts$ for $Y_p(\gl_2)$ is  equidimensional of dimension $p$.
\end{theorem}


\begin{proof}
 See \cite{FMO}.
 
\end{proof}


\begin{notation}
\label{notacPX}
For $P\in\k[x_1,x_2,\dots,x_n]$  and 
$\textbf{X}=\{x_{i_1},x_{i_1},\dots,x_{i_r}\}$ a set of varia\-bles, we will denote by $P^{\textbf{X}}$ the polynomial obtained from $P$ substituiting
 $$x_{i_1}=x_{i_2}=\cdots=x_{i_r}=0.$$ 
\end{notation}


\begin{proposition}
\label{W1s}
For all $1\leq s\leq p+1$ with $p>2$, the subvariety $\wk_{1s}$ of $\gts_1$ for $Y_p(\gl_3)$ is equidimensional of dimension $3p$.
\end{proposition}


\begin{proof}
Using the notation in \eqref{subvarWeak} the subvarieties $\wk_{1s}$'s are:

$$\wk_{1s}=V\left(\textbf{X}_s\cup\left\{p_{3j}\right\}_{j=1}^{2p}\right)\subset \k^{9p},\ \ \ s=1,2,\ldots,p,p+1,$$
where
 \begin{align*}
\textbf{X}_1    &=\bigcup_{i=1}^p\left\{X_{11}^{(i)},X_{12}^{(i)},X_{22}^{(i)}\right\}\cup\left\{X_{13}^{(i)}\right\}_{i=1}^p,\\
\textbf{X}_s    &=\bigcup_{i=1}^p\left\{X_{11}^{(i)},X_{12}^{(i)},X_{22}^{(i)}\right\}\cup\left\{X_{13}^{(i)}\right\}_{i=s}^p\cup\left\{X_{21}^{(i)}\right\}_{i=p-(s-2)}^p,\ \ s=2,3,\ldots,p,\\
\textbf{X}_{p+1}&=\bigcup_{i=1}^p\left\{X_{11}^{(i)},X_{12}^{(i)},X_{22}^{(i)}\right\}\cup\left\{X_{21}^{(i)}\right\}_{i=1}^p.\\
\end{align*}

We will prove that each $\wk_{1s}$ is equidimensional with
$$\dim\left(\wk_{1s}\right)=9p-(\underbrace{4p}_{=\left|\textbf{X}_s\right|}+2p)=3p.$$
In fact, by Corollary \eqref{regsubseq} and the Proposition \eqref{regvseq}, it is sufficient to prove that 

$$V\left(\textbf{X}_s\cup\left\{p_{3j}\right\}_{j=1}^{2p}\cup\textbf{Y}_s\right)\subset \k^{9p}$$  
with
 \begin{align*}
\textbf{Y}_1    &= \left\{X_{21}^{(i)}\right\}_{i=1}^{p}\cup\left\{X_{31}^{(i)}\right\}_{i=1}^p\\
\textbf{Y}_s    &= \left\{X_{13}^{(i)}\right\}_{i=1}^{s-1}\cup\left\{X_{21}^{(i)}\right\}_{i=1}^{p-(s-2)-1}\cup\left\{X_{31}^{(i)}\right\}_{i=1}^p,\ \ s=2,3,\ldots,p\\
\textbf{Y}_{p+1}&= \left\{X_{13}^{(i)}\right\}_{i=1}^{p}\cup\left\{X_{31}^{(i)}\right\}_{i=1}^p
\end{align*}
is equidimensional of dimension 
$9p-(\underbrace{4p}_{=\left|\textbf{X}_s\right|}+2p+\underbrace{2p}_{=\left|\textbf{Y}_s\right|})=p$.
Clearly, for each $s=1,2,\ldots,p,p+1$

$$V\left(\textbf{X}_s\cup\left\{p_{3j}\right\}_{j=1,}^{2p}\cup\textbf{Y}_s\right)=V\left(\textbf{Z}\cup\left\{p^{\textbf{Z}}_{3j}\right\}_{j=1}^{2p}\right)\subset \k^{9p},$$  
 where, by Proposition \eqref{equivGTs} and the Notation \eqref{notacPX}
 \begin{align*}
\textbf{Z}             &=\textbf{X}_s\cup\textbf{Y}_s= \bigcup_{i=1}^p\left\{X_{11}^{(i)},X_{12}^{(i)},X_{22}^{(i)},X_{13}^{(i)},X_{21}^{(i)},X_{31}^{(i)}\right\},\\
p_{31}^{\textbf{Z}}    &=X_{33}^{(1)},\\
p_{32}^{\textbf{Z}}    &= X_{33}^{(2)}-X_{23}^{(1)}X_{32}^{(1)}-\underbrace{X_{13}^{(1)}X_{31}^{(1)}}_{=0}\\
                       &=X_{33}^{(2)}-X_{23}^{(1)}X_{32}^{(1)},\\
p_{3i}^{\textbf{Z}}    &= -\sum_{r=1}^{i-2}\sum_{s=2}^{i-r-1}\underbrace{X_{13}^{(r)}X_{22}^{(s)}X_{31}^{(i-r-s)}}_{=0}+X_{33}^{(i)}-\sum_{s=1}^{i-1}\left(\underbrace{X_{13}^{(s)}X_{31}^{(i-s)}}_{=0}+X_{23}^{(s)}X_{32}^{(i-s)}\right)+\\
                       &\ \ \ +\sum_{r=1}^{i-2}\sum_{s=1}^{i-r-1}\biggl\{\underbrace{X_{12}^{(r)}X_{23}^{(s)}X_{31}^{(i-r-s)}+X_{13}^{(r)}X_{21}^{(s)}X_{32}^{(i-r-s)}}_{=0}\biggr\}\\
                       &= X_{33}^{(i)}-\sum_{s=1}^{i-1}X_{23}^{(s)}X_{32}^{(i-s)},\ \ \ i=3,4,\ldots,p, \\
p_{3\,p+1}^{\textbf{Z}}&=\sum_{r=1}^{p-1}\sum_{s=1}^{p-r}\biggl\{\underbrace{ X_{12}^{(r)}X_{23}^{(s)}X_{31}^{(p+1-r-s)}+X_{13}^{(r)}X_{21}^{(s)}X_{32}^{(p+1-r-s)}}_{=0}\biggr\}-\\
                       &\ \ \ -\sum_{s=1}^p\left(\underbrace{X_{13}^{(s)}X_{31}^{(p+1-s)}}_{=0}+ X_{23}^{(s)}X_{32}^{(p+1-s)}\right)-\sum_{r=1}^{p-1}\sum_{s=2}^{p-r}\underbrace{X_{13}^{(r)}X_{22}^{(s)}X_{31}^{(p+1-r-s)}}_{=0}\\
                       &= -\sum_{s=1}^pX_{23}^{(s)}X_{32}^{(p+1-s)},\\
p_{3\,p+i}^{\textbf{Z}}&=-\sum_{s=i}^pX_{23}^{(s)}X_{32}^{(p+i-s)}- \sum_{s=i}^p\underbrace{X_{13}^{(s)}X_{31}^{(p+i-s)}}_{=0}+\\
                       &\ \ \ +\sum_{r=1}^{i-1}\sum_{s=i-r}^p\biggl\{\underbrace{-X_{13}^{(r)}X_{22}^{(s)}X_{31}^{(p+i-r-s)}+X_{12}^{(r)}X_{23}^{(s)}X_{31}^{(p+i-r-s)}}_{=0}+\\
                       &\ \ \ +\underbrace{X_{13}^{(r)}X_{21}^{(s)}X_{32}^{(p+i-r-s)}}_{=0}\biggr\}+\\
                       &\ \ \ +\sum_{r=i}^{p}\sum_{s=1}^{p+i-r-1}\biggl\{\underbrace{-X_{13}^{(r)}X_{22}^{(s)}X_{31}^{(p+i-r-s)}+X_{12}^{(r)}X_{23}^{(s)}X_{31}^{(p+i-r-s)}}_{=0}+\\
                       &\ \ \ +\underbrace{X_{13}^{(r)}X_{21}^{(s)}X_{32}^{(p+i-r-s)}}_{=0}\biggr\}\\
                       &=-\sum_{s=i}^pX_{23}^{(s)}X_{32}^{(p+i-s)},\ \ \ \ \ i=2,3,\ldots,p.
\end{align*}

Now, projecting this variety on the variables
$$X_{i1}^{(t)},\ \ i=1,2,3;\ \ t=1,2,\ldots,p$$
and
$$X_{1j}^{(t)},\ \ j=1,2,3;\ \ t=1,2,\ldots,p$$
we have, by Propositions \eqref{regvseq} and \eqref{projhyper}, that to show the equidimensionality of 
$$V\left(\textbf{Z}\cup\left\{p^{\textbf{Z}}_{3j}\right\}_{j=1}^{2p}\right)\subset \k^{9p}$$  
with dimension $p$ is equivalent to prove that 
$$\widetilde{\wk}_{1}:=V\left(\textbf{A}\cup\left\{q_{3j}\right\}_{j=1}^{2p}\right)=V\left(\textbf{A}\cup\left\{p^{\textbf{Z}}_{3j}\right\}_{j=1}^{2p}\right)\subset \k^{9p-5p}=\k^{4p}$$   
is equidimensional of dimension 
 $$\dim\widetilde{\wk}_{1}=4p-(p+2p)=p,$$
where
\begin{align*}
\textbf{A}  &:=\left\{X_{22}^{(i)}\right\}_{i=1}^p, \\
q_{31}      &:=p_{31}^{\textbf{Z}}                
	      =X_{33}^{(1)}, \\
q_{3i}      &:=p_{3i}^{\textbf{Z}}
	      =X_{33}^{(i)}-\sum_{s=1}^{i-1}X_{23}^{(s)}X_{32}^{(i-s)},\ \ \ i=2,3,\ldots,p, \\
q_{3p+i}    &:=-p_{3p+i}^{\textbf{Z}}
	      =\sum_{s=i}^pX_{23}^{(s)}X_{32}^{(p+i-s)}, \ \ \ i=1,2,\ldots,p.
\end{align*}
Using the variable change $\varphi$
\begin{align*}
X_{22}^{(t)}&\longmapsto X_{11}^{(t)},\ \ \ t=1,2,\ldots,p\\
X_{23}^{(t)}&\longmapsto X_{12}^{(t)},\ \ \ t=1,2,\ldots,p\\
X_{32}^{(t)}&\longmapsto X_{21}^{(t)},\ \ \ t=1,2,\ldots,p\\
X_{33}^{(t)}&\longmapsto X_{22}^{(t)},\ \ \ t=1,2,\ldots,p
\end{align*}
we have $\widetilde{\wk}_{1}\cong\varphi\left(\widetilde{\wk}_{1}\right)$ and
 $$\varphi\left(\widetilde{\wk}_{1}\right)=\varphi\left(V\left(\textbf{A}\cup\left\{q_{3j}\right\}_{j=1}^{2p}\right)\right)=V\left(\varphi\left(\textbf{A}\right)\cup\left\{\varphi\left(q_{3j}\right)\right\}_{j=1}^{2p}\right)\subset \k^{4p},$$
with
\begin{align*}
\varphi\left(\textbf{A}\right)	&=\left\{\varphi\left(X_{22}^{(i)}\right)\right\}_{i=1}^p=\left\{X_{11}^{(i)}\right\}_{i=1}^p=\left\{p_{1i}\right\}_{i=1}^p,\\ 
\varphi\left(q_{31}\right)      &=\varphi\left(X_{33}^{(1)}\right)=X_{22}^{(1)}=p_{21},\\
\varphi\left(q_{3i}\right)      &=\varphi\left(X_{33}^{(i)}-\sum_{s=1}^{i-1}X_{23}^{(s)}X_{32}^{(i-s)}\right)=X_{22}^{(i)}- \sum_{s=1}^{i-1}X_{12}^{(s)}X_{21}^{(i-s)}\\
                                &=p_{2i},\ \ \ \mbox{for}\ \ i=2,3,\ldots,p\\
\varphi\left(q_{3\,p+i}\right)  &=\varphi\left(\sum_{s=i}^p X_{23}^{(s)}X_{32}^{(p+i-s)}\right)=\sum_{s=i}^p X_{12}^{(s)}X_{21}^{(p+i-s)}=p_{2\,p+i}, \ \ i=1,2,\ldots,p.
\end{align*}
Therefore, $\widetilde{\wk}_{1}\cong\varphi\left(\widetilde{\wk}_{1}\right)$ is the Gelfand-Tsetlin variety $\gts$ for Yangian $Y_p(\gl_2)$, which by Theorem \eqref{EquidGTsYp(2)} is equidimensional of dimension 
$$\dim\varphi\left(\widetilde{\wk}_{1}\right)=4p-3p=p.$$

\end{proof}


\begin{corollary}
\label{W1}
For $p>2$, the subvariety $\wk_{1}$ of $\gts_1$ for $Y_p(\gl_3)$ is equidimensional of dimension $3p$.
\end{corollary}


\begin{proof}
By Proposition \eqref{W1s} and Notation \eqref{subvarWeak}
$$\wk_{1}=\bigcup_{s=1}^{p+1}\wk_{1s}.$$

\end{proof}

\vspace{0.1cm}

\begin{proposition}
\label{W2s}
For all $2\leq s\leq p+1$ with $p>2$, the subvariety $\wk_{2s}$ of $\gts_1$ for $Y_p(\gl_3)$ is equidimensional of dimension $3p$.
\end{proposition}


\begin{proof}
Analogously to the proof of the Proposition \eqref{W1s}. Using the notation \eqref{subvarWeak} the subvarieties $\wk_{2s}$'s are:
 
 $$\wk_{2s}=V\left(\textbf{X}_s\cup\left\{p^{\textbf{X}_s}_{3j}\right\}_{j=1}^{2p}\right)\subset \k^{9p},\ \ \ s=2,3,\ldots,p,p+1,$$
 where
\begin{align*}
\textbf{X}_s    &=\bigcup_{i=1}^p\left\{X_{11}^{(i)},X_{12}^{(i)},X_{22}^{(i)}\right\}\cup\left\{X_{21}^{(i)}\right\}_{i=s}^p\cup\left\{X_{32}^{(i)}\right\}_{i=p-(s-2)}^p,\ \ s=2,3,\ldots,p\\
\textbf{X}_{p+1}&=\bigcup_{i=1}^p\left\{X_{11}^{(i)},X_{12}^{(i)},X_{22}^{(i)}\right\}_{i=1}^p\cup\left\{X_{32}^{(i)}\right\}_{i=1}^p.
\end{align*}

We will show that $\wk_{2s}$ is equidimensional of dimension 
$$\dim\left(\wk_{2s}\right)=9p-(\underbrace{4p}_{=\left|\textbf{X}_s\right|}+2p)=3p.$$
For this, by Corollary \eqref{regsubseq} and Proposition \eqref{regvseq} is enough to prove that

$$V\left(\textbf{X}_s\cup\left\{p^{\textbf{X}_s}_{3j}\right\}_{j=1,2,\ldots,2p}\cup\textbf{Y}_s\right)\subset \k^{9p},$$  
with
 \begin{align*}
\textbf{Y}_s    &= \left\{X_{21}^{(i)}\right\}_{i=1}^{s-1}\cup\left\{X_{32}^{(i)}\right\}_{i=1}^{p-(s-2)-1}\cup\left\{X_{23}^{(i)}\right\}_{i=1}^p,\ \ s=2,3,\ldots,p\\
\textbf{Y}_{p+1}&= \left\{X_{21}^{(i)}\right\}_{i=1}^{p}\cup\left\{X_{23}^{(i)}\right\}_{i=1}^p
\end{align*}
is equidimensional with dimension 
$9p-(\underbrace{4p}_{=\left|\textbf{X}_s\right|}+2p+\underbrace{2p}_{=\left|\textbf{Y}_s\right|})=p$. 
Clearly, for each $s=2,3,\ldots,p,p+1$

$$V\left(\textbf{X}_s\cup\left\{p^{\textbf{X}_s}_{3j}\right\}_{j=1}^{2p}\cup\textbf{Y}_s\right)=V\left(\textbf{Z}\cup\left\{p^{\textbf{Z}}_{3j}\right\}_{j=1}^{2p}\right)\subset \k^{9p},$$  
where, by Proposition \eqref{equivGTs} and Notation \eqref{notacPX}
 \begin{align*}
\textbf{Z}             &:=\textbf{X}_s\cup \textbf{Y}_s=\bigcup_{i=1}^p\left\{X_{11}^{(i)},X_{12}^{(i)},X_{22}^{(i)},X_{21}^{(i)},X_{23}^{(i)},X_{32}^{(i)}\right\},\\
p_{31}^{\textbf{Z}}    &=X_{33}^{(1)},\\
p_{3i}^{\textbf{Z}}    &=X_{33}^{(i)}-\sum_{s=1}^{i-1}X_{13}^{(s)}X_{31}^{(i-s)},\ \ \ i=2,3,\ldots,p, \\
p_{3\,p+i}^{\textbf{Z}}&=-\sum_{s=i}^pX_{13}^{(s)}X_{31}^{(p+i-s)},\ \ \ \ \ i=1,2,\ldots,p.
\end{align*}
Projecting this variety on the variables 

$$X_{i2}^{(t)},\ \ i=1,2,3;\ \ t=1,2,\ldots,p$$
and
$$X_{2j}^{(t)},\ \ j=1,2,3;\ \ t=1,2,\ldots,p$$
we have, by Propositions \eqref{regvseq} and \eqref{projhyper} that to show the equidimensionality of
$$V\left(\textbf{Z}\cup\left\{p^{\textbf{Z}}_{3j}\right\}_{j=1}^{2p}\right)\subset \k^{9p}$$
of dimension $p$ is equivalent to prove that 
$$\widetilde{\wk}_{2}:=V\left(\textbf{A}\cup\left\{q_{3j}\right\}_{j=1}^{2p}\right)=V\left(\textbf{A}\cup\left\{p^{\textbf{Z}}_{3j}\right\}_{j=1}^{2p}\right)\subset \k^{9p-5p}=\k^{4p}$$   
is equidimensional of dimension 
 $$\dim\widetilde{\wk}_{2}=4p-(p+2p)=p,$$
where
\begin{align*}
\textbf{A}&:=\left\{X_{11}^{(i)}\right\}_{i=1}^p,\\
q_{31}    &:=p_{31}^{\textbf{Z}}=X_{33}^{(1)},\\
q_{3i}    &:=p_{3i}^{\textbf{Z}}=X_{33}^{(i)}-\sum_{s=1}^{i-1}X_{13}^{(s)}X_{31}^{(i-s)},\ \ \ i=2,3,\ldots,p, \\
q_{3p+i}  &:=-p_{3p+i}^{\textbf{Z}}=\sum_{s=i}^pX_{13}^{(s)}X_{31}^{(p+i-s)}, \ \ i=1,2,\ldots,p
\end{align*}
and using the variable change $\varphi$
\begin{align*}
X_{11}^{(t)}&\longmapsto X_{11}^{(t)},\ \ \ t=1,2,\ldots,p\\
X_{13}^{(t)}&\longmapsto X_{12}^{(t)},\ \ \ t=1,2,\ldots,p\\
X_{31}^{(t)}&\longmapsto X_{21}^{(t)},\ \ \ t=1,2,\ldots,p\\
X_{33}^{(t)}&\longmapsto X_{22}^{(t)},\ \ \ t=1,2,\ldots,p,
\end{align*}
we have $\widetilde{\wk}_{2}\cong\varphi\left(\widetilde{\wk}_{2}\right)$ and
$$\varphi\left(\widetilde{\wk}_{2}\right)=\varphi\left(V\left(\textbf{A}\cup\left\{q_{3j}\right\}_{j=1}^{2p}\right)\right)=V\left(\varphi\left(\textbf{A}\right)\cup\left\{\varphi\left(q_{3j}\right)\right\}_{j=1}^{2p}\right)\subset k^{4p},$$
with
\begin{align*}
\varphi\left(\textbf{A}\right)&=\left\{X_{11}^{(i)}\right\}_{i=1}^p=\left\{p_{1i}\right\}_{i=1}^p,\\
\varphi\left(q_{31}\right)    &=\varphi\left(X_{33}^{(1)}\right)=X_{22}^{(1)}=p_{21},\\
\varphi\left(q_{3i}\right)    &=\varphi\left(X_{33}^{(i)}-\sum_{s=1}^{i-1}X_{13}^{(s)}X_{31}^{(i-s)}\right)=X_{22}^{(i)}-\sum_{s=1}^{i-1}X_{12}^{(s)}X_{21}^{(i-s)}=p_{2i},\\
                               &\ \ \ \mbox{for}\ \ i=2,3,\ldots,p,		      \\
\varphi\left(q_{3\,p+i}\right)&=\varphi\left(\sum_{s=i}^pX_{13}^{(s)}X_{31}^{(p+i-s)}\right)
                               =\sum_{s=i}^pX_{12}^{(s)}X_{21}^{(p+i-s)}
                               =p_{2\,p+i}, \ \ i=1,2,\ldots,p.\\
\end{align*}
Therefore, $\widetilde{\wk}_{2}\cong\varphi\left(\widetilde{\wk}_{2}\right)$ is Gelfand-Tsetlin variety $\gts$ for Yangian $Y_p(\gl_2)$, which by
Theorem \eqref{EquidGTsYp(2)} is equidimensional of dimension 
$$\dim\varphi\left(\widetilde{\wk}_{2}\right)=4p-3p=p.$$

\end{proof}


\begin{corollary}
\label{W2}
For $p>2$, the subvariety $\wk_{2}$ of $\gts_1$ for $Y_p(\gl_3)$ is equidimensional with dimension $3p$.
\end{corollary}


\begin{proof}
By Proposition \eqref{W2s} and Notation \eqref{subvarWeak}
$$\wk_{2}=\bigcup_{s=2}^{p+1}\wk_{2s}.$$

\end{proof}

\vspace{0.4cm}

Now, we can prove the main result.

\vspace{0.5cm}

\begin{theorem}
\label{G1}
The variety $\gts_1$ for $Y_p(\gl_3)$ is equidimensional of dimension 
$$\dim\gts_1=3p.$$
\end{theorem}


\begin{proof}
We will prove by induction over $p$. 

By Theorem \eqref{fracagln} or Proposition \eqref{Y1(3)}, we have that $\gts_1$ for $Y_1(\gl_3)$ is equidimensional with dimension 
$$\dim\gts_1=3.$$
Similarly, of the Proposition \eqref{Y2(3)} follows that $\gts_1$ for $Y_2(\gl_3)$ is equidimensional of dimension 
$$\dim\gts_1=6.$$
Now suppose $p>2$ and that the variety $\gts_1$ for $Y_{p-1}(\gl_3)$ is equidimensional with dimension $3(p-1)$.

As a consequence of the Corollary \eqref{decompweak}, we have the decomposition 
 $$\gts_1=\wk^{p}\cup\wk_{1}\cup\wk_{2},\ \ \mbox{for}\ \ p\geq 3.$$
By corollaries \eqref{W1} and \eqref{W2}, the components $\wk_{1}$ and $\wk_{2}$ are equidimensionals of dimension $3p$. 

We will show that $\wk^p$ is equidimensional with dimension $3p$. Using the notation \eqref{subvarWeak}, we have
 $$\wk^p=V\left(\textbf{X}\cup\left\{p_{3j}\right\}_{j=1}^{3p-3}\right)\subset \k^{9p},$$
where
$$\textbf{X}=\bigcup_{i=1}^p\left\{X_{11}^{(i)},X_{12}^{(i)},X_{22}^{(i)}\right\}\cup\left\{X_{32}^{(p)},X_{21}^{(p)},X_{13}^{(p)}\right\}.$$
To prove that $\wk^p$ is equidimensional of $\dim\left(\wk^p\right)=3p$, by Corollary \eqref{regsubseq} and Proposition \eqref{regvseq}, is sufficient to show that 

$$V\left(\textbf{X}\cup\left\{p_{3j}\right\}_{j=1,2}^{3p-3}\cup\left\{X_{23}^{(p)},X_{31}^{(p)},X_{33}^{(p)}\right\}\right)\subset \k^{9p}$$
is equidimensional with dimension $9p-(6p+3)=3p-3=3(p-1)$. Clearly

$$V\left(\textbf{Y}\cup\left\{p^{\textbf{Y}}_{3j}\right\}_{j=1}^{3p-3}\right)=V\left(\textbf{X}\cup\left\{p_{3j}\right\}_{j=1}^{3p-3}\cup\left\{X_{23}^{(p)},X_{31}^{(p)},X_{33}^{(p)}\right\}\right)\subset \k^{9p},$$
where, as a consequence of the Proposition \eqref{equivGTs} and Notation \eqref{notacPX}
\begin{align*}
\textbf{Y}         &=\bigcup_{i=1}^{p-1}\left\{X_{11}^{(i)},X_{12}^{(i)},X_{22}^{(i)}\right\}\cup\left\{X_{ij}^{(p)}\right\}_{i,j=1}^{3},\\
p_{31}^{\textbf{Y}}&=X_{33}^{(1)},\\
p_{32}^{\textbf{Y}}&=X_{33}^{(2)}-X_{23}^{(1)}X_{32}^{(1)}-X_{13}^{(1)}X_{31}^{(1)},
\end{align*}
for every $i=3,4,\ldots,p-1$ we have
\begin{align*}
p_{3i}^{\textbf{Y}}&=-\sum_{r=1}^{i-2}\sum_{s=2}^{i-r-1}\underbrace{X_{13}^{(r)}X_{22}^{(s)}X_{31}^{(i-r-s)}}_{=0}+X_{33}^{(i)}-\sum_{s=1}^{i-1}\left(X_{13}^{(s)}X_{31}^{(i-s)}+X_{23}^{(s)}X_{32}^{(i-s)}\right)+\\
		   &\ \ \ +\sum_{r=1}^{i-2}\sum_{s=1}^{i-r-1}\biggl\{\underbrace{X_{12}^{(r)}X_{23}^{(s)}X_{31}^{(i-r-s)}}_{=0}+X_{13}^{(r)}X_{21}^{(s)}X_{32}^{(i-r-s)}\biggr\}\\
		   &= X_{33}^{(i)}-\sum_{s=1}^{i-1}\left(X_{13}^{(s)}X_{31}^{(i-s)}+X_{23}^{(s)}X_{32}^{(i-s)}\right)+\sum_{r=1}^{i-2}\sum_{s=1}^{i-r-1}X_{13}^{(r)}X_{21}^{(s)}X_{32}^{(i-r-s)},
\end{align*}
analogously, 
\begin{align*}
p_{3p}^{\textbf{Y}}    &= -\sum_{r=1}^{p-2}\sum_{s=2}^{p-r-1}\underbrace{X_{13}^{(r)}X_{22}^{(s)}X_{31}^{(p-r-s)}}_{=0}+\underbrace{X_{33}^{(p)}}_{=0}-\sum_{s=1}^{p-1}\left(X_{13}^{(s)}X_{31}^{(p-s)}+X_{23}^{(s)}X_{32}^{(p-s)}\right)+\\
                       &\ \ \ +\sum_{r=1}^{p-2}\sum_{s=1}^{p-r-1}\biggl\{\underbrace{X_{12}^{(r)}X_{23}^{(s)}X_{31}^{(p-r-s)}}_{=0}+X_{13}^{(r)}X_{21}^{(s)}X_{32}^{(p-r-s)}\biggr\}\\
                       &= -\sum_{s=1}^{p-1}\left(X_{13}^{(s)}X_{31}^{(p-s)}+X_{23}^{(s)}X_{32}^{(p-s)}\right)+\sum_{r=1}^{p-2}\sum_{s=1}^{p-r-1}X_{13}^{(r)}X_{21}^{(s)}X_{32}^{(p-r-s)},\\
p_{3\,p+1}^{\textbf{Y}}&=\sum_{r=1}^{p-1}\sum_{s=1}^{p-r}\biggl\{ \underbrace{X_{12}^{(r)}X_{23}^{(s)}X_{31}^{(p+1-r-s)}}_{=0}+X_{13}^{(r)}X_{21}^{(s)}X_{32}^{(p+1-r-s)}\biggr\}-\\
                       &\ \ \ -\sum_{s=1}^p\left(X_{13}^{(s)}X_{31}^{(p+1-s)}+ X_{23}^{(s)}X_{32}^{(p+1-s)}\right)-\sum_{r=1}^{p-1}\sum_{s=2}^{p-r}\underbrace{X_{13}^{(r)}X_{22}^{(s)}X_{31}^{(p+1-r-s)}}_{=0}\\
                       &= \sum_{r=1}^{p-1}\sum_{s=1}^{p-r}X_{13}^{(r)}X_{21}^{(s)}X_{32}^{(p+1-r-s)}-\sum_{s=2}^{p-1}\left(X_{13}^{(s)}X_{31}^{(p+1-s)}+X_{23}^{(s)}X_{32}^{(p+1-s)}\right)-\\
                       &\ \ \ -\left(\underbrace{X_{13}^{(1)}X_{31}^{(p)}+X_{23}^{(1)}X_{32}^{(p)}+X_{13}^{(p)}X_{31}^{(1)}+X_{23}^{(p)}X_{32}^{(1)}}_{=0}\right)\\
                       &= \sum_{r=1}^{p-1}\sum_{s=1}^{p-r}X_{13}^{(r)}X_{21}^{(s)}X_{32}^{(p+1-r-s)}-\sum_{s=2}^{p-1}\left(X_{13}^{(s)}X_{31}^{(p+1-s)}+X_{23}^{(s)}X_{32}^{(p+1-s)}\right),
\end{align*}
for each $i=2,3,\ldots,p-2$ 
\begin{align*}
p_{3\,p+i}^{\textbf{Y}}&=-\sum_{s=i}^pX_{23}^{(s)}X_{32}^{(p+i-s)}- \sum_{s=i}^pX_{13}^{(s)}X_{31}^{(p+i-s)}+\\
                       &\ \ \ +\sum_{r=1}^{i-1}\sum_{s=i-r}^p\biggl\{-X_{13}^{(r)}X_{22}^{(s)}X_{31}^{(p+i-r-s)}+X_{12}^{(r)}X_{23}^{(s)}X_{31}^{(p+i-r-s)}+\\
                       &\ \ \ +X_{13}^{(r)}X_{21}^{(s)}X_{32}^{(p+i-r-s)}\biggr\}+\\
                       &\ \ \ +\sum_{r=i}^{p}\sum_{s=1}^{p+i-r-1}\biggl\{-X_{13}^{(r)}X_{22}^{(s)}X_{31}^{(p+i-r-s)}+X_{12}^{(r)}X_{23}^{(s)}X_{31}^{(p+i-r-s)}+\\
                       &\ \ \ +X_{13}^{(r)}X_{21}^{(s)}X_{32}^{(p+i-r-s)}\biggr\}\\
                       &= -\sum_{s=i+1}^{p-1}X_{23}^{(s)}X_{32}^{(p+i-s)}\underbrace{-X_{23}^{(i)}X_{32}^{(p)}-X_{23}^{(p)}X_{32}^{(i)}}_{=0}-\sum_{s=i+1}^{p-1}X_{13}^{(s)}X_{31}^{(p+i-s)}-\\
                       &\ \ \ \underbrace{-X_{13}^{(i)}X_{31}^{(p)}-X_{13}^{(p)}X_{31}^{(i)}}_{=0}+\\
                       &\ \ \ +\sum_{r=1}^{i-1}\sum_{s=i-r+1}^{p-1}\biggl\{\underbrace{-X_{13}^{(r)}X_{22}^{(s)}X_{31}^{(p+i-r-s)}+X_{12}^{(r)}X_{23}^{(s)}X_{31}^{(p+i-r-s)}}_{=0}+\\
                       &\ \ \ +X_{13}^{(r)}X_{21}^{(s)}X_{32}^{(p+i-r-s)}\biggr\}+\\
                       &\ \ \ +\sum_{r=1}^{i-1}\biggl\{\underbrace{-X_{13}^{(r)}X_{22}^{(i-r)}X_{31}^{(p)}+X_{12}^{(r)}X_{23}^{(i-r)}X_{31}^{(p)}+X_{13}^{(r)}X_{21}^{(i-r)}X_{32}^{(p)}}_{=0}\biggr\}+\\
                       &\ \ \ +\sum_{r=1}^{i-1}\biggl\{\underbrace{-X_{13}^{(r)}X_{22}^{(p)}X_{31}^{(i-r)}+X_{12}^{(r)}X_{23}^{(p)}X_{31}^{(i-r)}+X_{13}^{(r)}X_{21}^{(p)}X_{32}^{(i-r)}}_{=0}\biggr\}+\\
                       &\ \ \ +\sum_{r=i}^{p-1}\sum_{s=1}^{p+i-r-1}\biggl\{\underbrace{-X_{13}^{(r)}X_{22}^{(s)}X_{31}^{(p+i-r-s)}+X_{12}^{(r)}X_{23}^{(s)}X_{31}^{(p+i-r-s)}}_{=0}+\\
                       &\ \ \ +X_{13}^{(r)}X_{21}^{(s)}X_{32}^{(p+i-r-s)}\biggr\}+\\
                       &\ \ \ +\sum_{s=1}^{i-1}\biggl\{\underbrace{-X_{13}^{(p)}X_{22}^{(s)}X_{31}^{(i-s)}+X_{12}^{(p)}X_{23}^{(s)}X_{31}^{(i-s)}+X_{13}^{(p)}X_{21}^{(s)}X_{32}^{(i-s)}}_{=0}\biggr\},\\
                       &= -\sum_{s=i+1}^{p-1}X_{23}^{(s)}X_{32}^{(p+i-s)}-\sum_{s=i+1}^{p-1}X_{13}^{(s)}X_{31}^{(p+i-s)}+\\
                       &\ \ \ +\sum_{r=1}^{i-1}\sum_{s=i-r+1}^{p-1}X_{13}^{(r)}X_{21}^{(s)}X_{32}^{(p+i-r-s)}+\sum_{r=i}^{p-1}\sum_{s=1}^{p+i-r-1}X_{13}^{(r)}X_{21}^{(s)}X_{32}^{(p+i-r-s)},
\end{align*}
continuing with the same argument
\begin{align*}
p_{3\,2p-1}^{\textbf{Y}}&= \underbrace{-\sum_{s=p-1}^pX_{23}^{(s)}X_{32}^{(2p-1-s)}-\sum_{s=p-1}^pX_{13}^{(s)}X_{31}^{(2p-1-s)}}_{=0}+\\
                        &\ \ \ +\sum_{r=1}^{p-2}\sum_{s=p-1-r}^p\biggl\{\underbrace{-X_{13}^{(r)}X_{22}^{(s)}X_{31}^{(2p-1-r-s)}+X_{12}^{(r)}X_{23}^{(s)}X_{31}^{(2p-1-r-s)}}_{=0}+\\
                        &\ \ \ +X_{13}^{(r)}X_{21}^{(s)}X_{32}^{(2p-1-r-s)}\biggr\}+\\
                        &\ \ \ +\sum_{r=p-1}^{p}\sum_{s=1}^{2p-r-2}\biggl\{\underbrace{-X_{13}^{(r)}X_{22}^{(s)}X_{31}^{(2p-1-r-s)}+X_{12}^{(r)}X_{23}^{(s)}X_{31}^{(2p-1-r-s)}}_{=0}+\\
                        &\ \ \ +X_{13}^{(r)}X_{21}^{(s)}X_{32}^{(2p-1-r-s)}\biggr\}\\
                        &= \sum_{r=1}^{p-2}\sum_{s=p-r}^{p-1}X_{13}^{(r)}X_{21}^{(s)}X_{32}^{(2p-1-r-s)}+\underbrace{\sum_{r=1}^{p-2}X_{13}^{(r)}X_{21}^{(p-1-r)}X_{32}^{(p)}}_{=0}+\\
                        &\ \ \ +\underbrace{\sum_{r=1}^{p-2}X_{13}^{(r)}X_{21}^{(p)}X_{32}^{(p-1-r)}}_{=0}+\sum_{s=1}^{p-1}X_{13}^{(p-1)}X_{21}^{(s)}X_{32}^{(p-s)}+\\
                        &\ \ \ +\sum_{s=1}^{p-2}\underbrace{X_{13}^{(p)}X_{21}^{(s)}X_{32}^{(p-1-s)}}_{=0}\\
                        &= \sum_{r=1}^{p-2}\sum_{s=p-r}^{p-1}X_{13}^{(r)}X_{21}^{(s)}X_{32}^{(2p-1-r-s)}+\sum_{s=1}^{p-1}X_{13}^{(p-1)}X_{21}^{(s)}X_{32}^{(p-s)},\\
p_{3\,2p}^{\textbf{Y}}  &= \underbrace{-X_{23}^{(p)}X_{32}^{(p)}-X_{13}^{(p)}X_{31}^{(p)}}_{=0}+\\
                        &\ \ \ +\sum_{r=1}^{p-1}\sum_{s=p-r}^p\biggl\{\underbrace{-X_{13}^{(r)}X_{22}^{(s)}X_{31}^{(2p-r-s)}+X_{12}^{(r)}X_{23}^{(s)}X_{31}^{(2p-r-s)}}_{=0}+\\
                        &\ \ \ +X_{13}^{(r)}X_{21}^{(s)}X_{32}^{(2p-r-s)}\biggr\}+\\
                        &\ \ \ +\sum_{s=1}^{p-1}\biggl\{\underbrace{-X_{13}^{(p)}X_{22}^{(s)}X_{31}^{(p-s)}+X_{12}^{(p)}X_{23}^{(s)}X_{31}^{(p-s)}+X_{13}^{(p)}X_{21}^{(s)}X_{32}^{(p-s)}}_{=0}\biggr\}\\
                        &= \sum_{r=1}^{p-1}\sum_{s=p-r+1}^{p-1}X_{13}^{(r)}X_{21}^{(s)}X_{32}^{(2p-r-s)}+\underbrace{\sum_{r=1}^{p-1}X_{13}^{(r)}X_{21}^{(p-r)}X_{32}^{(p)}}_{=0}+\\
                        &\ \ \ +\underbrace{\sum_{r=1}^{p-1}X_{13}^{(r)}X_{21}^{(p)}X_{32}^{(p-r)}}_{=0}\\
                        &= \sum_{r=1}^{p-1}\sum_{s=p-r+1}^{p-1}X_{13}^{(r)}X_{21}^{(s)}X_{32}^{(2p-r-s)},
\end{align*}
for each $i=1,2,\ldots,p-4$ we have
\begin{align*}
p_{3\,2p+i}^{\textbf{Y}}&=\sum_{r=i}^{p}\sum_{s=p+i-r}^{p} \biggl\{\underbrace{-X_{13}^{(r)}X_{22}^{(s)}X_{31}^{(2p+i-r-s)}+X_{12}^{(r)}X_{23}^{(s)}X_{31}^{(2p+i-r-s)}}_{=0}+\\
						&\ \ \ +X_{13}^{(r)}X_{21}^{(s)}X_{32}^{(2p+i-r-s)}\biggr\}\\
						&= \sum_{r=i+1}^{p-1}\sum_{s=p+i-r}^{p}X_{13}^{(r)}X_{21}^{(s)}X_{32}^{(2p+i-r-s)}+\underbrace{X_{13}^{(i)}X_{21}^{(p)}X_{32}^{(p)}}_{=0}+\\
						&\ \ \ +\underbrace{\sum_{s=i}^{p} X_{13}^{(p)}X_{21}^{(s)}X_{32}^{(p+i-s)}}_{=0}\\
						&= \sum_{r=i+1}^{p-1}\sum_{s=p+i-r+1}^{p-1}X_{13}^{(r)}X_{21}^{(s)}X_{32}^{(2p+i-r-s)}\underbrace{+\sum_{r=i+1}^{p-1}X_{13}^{(r)}X_{21}^{(p+i-r)}X_{32}^{(p)}}_{=0}+\\
						&\ \ \ +\underbrace{\sum_{r=i+1}^{p-1}X_{13}^{(r)}X_{21}^{(p)}X_{32}^{(p+i-r)}}_{=0}\\
						&= \sum_{r=i+1}^{p-1}\sum_{s=p+i-r+1}^{p-1}X_{13}^{(r)}X_{21}^{(s)}X_{32}^{(2p+i-r-s)},
\end{align*}
analogously
\begin{align*}
p_{3\,3p-3}^{\textbf{Y}}&=\sum_{r=p-3}^{p}\sum_{s=2p-3-r}^{p} \biggl\{\underbrace{-X_{13}^{(r)}X_{22}^{(s)}X_{31}^{(3p-3-r-s)}+X_{12}^{(r)}X_{23}^{(s)}X_{31}^{(3p-3-r-s)}}_{=0}+\\
						&\ \ \ +X_{13}^{(r)}X_{21}^{(s)}X_{32}^{(3p-3-r-s)}\biggr\}\\
                        &= \underbrace{X_{13}^{(p-3)}X_{21}^{(p)}X_{32}^{(p)}+\sum_{s=p-1}^{p} X_{13}^{(p-2)}X_{21}^{(s)}X_{32}^{(2p-1-s)}}_{=0}+\\
                        &\ \ \ +\sum_{s=p-2}^{p} X_{13}^{(p-1)}X_{21}^{(s)}X_{32}^{(2p-2-s)}+\underbrace{\sum_{s=p-3}^{p} X_{13}^{(p)}X_{21}^{(s)}X_{32}^{(2p-3-s)}}_{=0}\\
                        &= \underbrace{X_{13}^{(p-1)}X_{21}^{(p-2)}X_{32}^{(p)}}_{=0}+X_{13}^{(p-1)}X_{21}^{(p-1)}X_{32}^{(p-1)}+\underbrace{X_{13}^{(p-1)}X_{21}^{(p)}X_{32}^{(p-2)}}_{=0}\\
                        &= X_{13}^{(p-1)}X_{21}^{(p-1)}X_{32}^{(p-1)}.\\
\end{align*}
 Now, projecting the variety on the variables
$$X_{ij}^{(p)},\ \ i,j=1,2,3,$$
we have, by Proposition \eqref{projhyper} that to show the equidimensionality of 
$$V\left(\textbf{Y}\cup\left\{p^{\textbf{Y}}_{3j}\right\}_{j=1}^{3p-3}\right)\subset \k^{9p}$$
of dimension $3(p-1)$, is equivalent to prove that  

$$\widetilde{\wk}^p=V\left(\textbf{Z}\cup\left\{q^{\textbf{Z}}_{3j}\right\}_{j=1}^{3p-3}\right)\subset \k^{9p-9},$$
with 

$$\textbf{Z}=\left\{X_{11}^{(i)}\right\}_{i=1}^{p-1}\cup \left\{X_{12}^{(i)}\right\}_{i=1}^{p-1}\cup\left\{X_{22}^{(i)}\right\}_{i=1}^{p-1}$$
and
$$q^{\textbf{Z}}_{3j}=p^{\textbf{Y}}_{3j}$$
is equidimensional with $\dim\left(\widetilde{\wk}^p\right)=9p-9-6(p-1)=3(p-1)$. But, we note that $\widetilde{\wk}^p$ is the weak version $\gts_1$ to 
$Y_{p-1}(\gl_3)$.

\end{proof}

\end{document}